\documentclass[english]{amsart}

\usepackage[all,cmtip]{xy}
\usepackage{hyperref}
\usepackage{blindtext}
\usepackage[T1]{fontenc}
\usepackage[utf8]{inputenc}
\usepackage{textcomp}

\usepackage{geometry}
	\usepackage{tikz, tikz-cd}
\usepackage{amsmath,amssymb,amscd,amsfonts}

\usepackage{babel}
\usepackage{amstext}
\usepackage{amsmath}
\usepackage{amsfonts}
\usepackage{latexsym}
\usepackage{ifthen}

\usepackage[all,cmtip]{xy}
\xyoption{all}
\usepackage{enumitem}
\setlist[enumerate]{label=(\thethm.\arabic*), before={\setcounter{enumi}{\value{equation}}}, after={\setcounter{equation}{\value{enumi}}}}

\newcommand{\CC}{\mathbb{C}}

\newcommand{\DD}{\mathbb D}

\newcommand{\ddbar}{\partial\bar{\partial}}

\newcommand{\wh}{\widehat}

\newcommand{\cX}{\mathcal{X}}
\newcommand{\cF}{\mathcal{F}}

\newcommand{\Lie}{\mathcal{L}}
\newcommand{\cC}{\mathcal{C}}
\newcommand{\cH}{\mathcal{H}}

\newcommand{\cI}{\mathcal{I}}
\newcommand{\cG}{\mathcal{G}}

\renewcommand{\O}{\mathcal{O}}

\newcommand{\ep}{\varepsilon}
\renewcommand{\epsilon}{\varepsilon}
\newcommand{\im}{\mathrm{Im} \,}

\renewcommand{\ker}{\mathrm{Ker} \,}

\newcommand{\ol}{\overline}

\renewcommand{\ge}{\geqslant}
\renewcommand{\le}{\leqslant}
\renewcommand{\leq}{\leqslant}
\renewcommand{\geq}{\geqslant}

\newcommand{\Div}{\mathrm{Div}}

\newcommand{\Imm}{\mathrm{Im} \,}

\newcommand{\ddc}{dd^c}

\newcommand{\Pic}{\mathrm{Pic}}

\newcommand{\Rel}{\mathrm{Re}}
\newcommand{\Ker}{\mathrm{Ker}}
\newcommand{\Res}{\mathrm{Res}}

\newcommand{\dbar}{\bar \partial}

\newtheorem{thm}{Theorem}[section]
\newtheorem{lemme}[thm]{Lemma}
\newtheorem{proposition}[thm]{Proposition}
\newtheorem{quest}[thm]{Question}

\newtheorem{exemple}[thm]{Example}
\newtheorem{conjecture}[thm]{Conjecture}
\newtheorem{defn}[thm]{Definition}
\newtheorem{cor}[thm]{Corollary}

\newtheorem{remark}[thm]{Remark}

\numberwithin{equation}{thm}

\title{Hodge theory for \\ twisted log-differential forms}

\author[J. Cao]{Junyan Cao}
\address{Laboratoire de Mathématiques J.A. Dieudonné UMR 7351 CNRS, Université Côte d'Azur Parc Valrose 06108, Nice, France} 
\email{junyan.cao@univ-cotedazur.fr; junyan.cao@unice.fr}
\urladdr{https://sites.google.com/site/junyancao}

\begin{document} 
	
	\maketitle

Hodge theory plays a central role in complex geometry, furnishing a profound and elegant correspondence among the analytic, topological, and algebraic aspects of complex manifolds. By identifying cohomological invariants with spaces of harmonic forms, it yields deep structural insights and powerful consequences across the field.

In parallel, \(L^2\) estimates --arising in the study of holomorphic line bundles endowed with possibly singular metrics-- have emerged as indispensable analytic tools in complex geometry. They underpin fundamental results in vanishing theorems, extension phenomena, and the analysis of singular spaces.

It is thus both natural and fruitful to bring these two perspectives together, extending Hodge theory to differential forms with values in bundles equipped with singular metrics. In this broader setting, techniques from 
\(L^2\)-theory prove to be remarkably effective. In this survey, we review recent developments along these lines, drawing on the works \cite{JCMP, CP-inv, CDHP}. The applications discussed in the three final sections both illustrate the breadth of these advances and point toward promising directions for future investigation.  
	
	\medskip
	
	We next briefly explain the results and the organization of this article.
	
	\medskip
	
	In Section~1, we discuss a version of Hodge theory for line bundles equipped with singular metrics over compact Kähler manifolds. 
	Let \(X\) be a compact K\"ahler manifold, and let \(L \to X\) be a holomorphic line bundle. Assume that \(L\) admits a singular metric \(h_L\) such that the corresponding curvature current can be written as follows
	\begin{equation}\label{singmetrin}
		i\Theta_{h_L}(L)=\sum a_i [D_i]+\theta,
	\end{equation}
	where \(D:=\sum D_i\) is an SNC divisor, \(a_i\in \mathbb Q\), and \(\theta\) is a smooth form on \(X\).
	
	Let \(m\in\mathbb N\) be a sufficiently large and divisible integer. We consider the orbifold structure \((X, (1 -\frac{1}{m}) D)\) and a metric \(\omega_\cC\) adapted to it. In \cite{JCMP}, we developed a Hodge theory for \((X,L,h_L, \omega_\cC)\); see Theorem \ref{CoHo1}. This can be seen as a natural generalisation of the results obtained in the context of orbifolds
	(i.e. in the absence of the additional twisting with $L$). 
	One of the main applications of the theory is the crucial \(\ddbar\)-lemma in conic sense (established under natural assumptions on the form \(\theta\) above); see Theorem~\ref{conicddabr}.
	
	An important question in this framework is the following: assume that some \(L\)-valued form is \(\dbar\)-exact in conic sense. Under which conditions is this form \(\dbar\)-exact in the usual sense? We refer to Theorem~\ref{classsol} for the precise meaning of the question and a result in this direction.
	
	\medskip
	
	In Section~2, we generalize the above construction to the quasi-compact setting, following \cite{CDHP}. More precisely, let \(X\) be a compact K\"ahler manifold, and let \(D=\sum_{i=1}^k D_i\) be an SNC divisor on \(X\). We would like to study  the logarithmic holomorphic flat bundles \((L,\nabla)\) on $(X, D)$, namely the pairs consisting on holomorphic line bundles \(L\) on \(X\) together with a logarithmic holomorphic connection
	\begin{equation}\label{holcon}
		\nabla: L \to L\otimes \Omega_X^{1}(\log D),
	\end{equation}
	such that \(\nabla^2=0\) on \(X\setminus D\).
	
	One can always construct a harmonic metric \(h_L\) on \(L\), so that the equality \(i\Theta_{h_L}(L)=0\) holds on \(X\setminus D\). The metric \(h_L\) may have singularities along \(D\), but one can always assume that the Lelong numbers $\nu (L, h_L)$ of $h_L$ along \(D\) are rational.
	
	Let \(\Delta_2 \subset D\) be the union of components such that the equality
	$$\Res_{D_i}(L, \nabla) + \nu_{D_i} (L, h_L) \neq 0$$
holds.	 
In this setting the ambient metric on \(X\) we are working with will have conic singularities along \(D\setminus \Delta_2\) and Poincaré type singularities along \(\Delta_2\). Notice that
such a metric is complete near the boundary \(\Delta_2\), and also that the singularities are imposed by geometric/analytic considerations.
	
	With respect to this rather involved setting, our first task would be to define what is meant by a ``smooth form'' see Definition~\ref{smoothK}. This turns out to be a suitable notion, given that we obtain a strong Hodge decomposition Theorem~\ref{smoothhodge}. Moreover, the space we define provides a natural fine resolution of the complex induced by \eqref{holcon}, see Subsection \ref{resolutionsection}. 
	
	\medskip
	
	Section~3 is devoted to generalisation of some of the results of the previous two sections by replacing differential forms with currents. The idea of this generalisation is as follows.
	For many applications, we need to consider forms with logarithmic poles along some divisors \(E\), where \(E+D\) is snc. Thanks to an idea due to Noguchi~\cite{Nog95}, one can embed logarithmic forms into the dual space, namely the space of currents. We adapt this approach to the settings of Sections~1 and~2.
	
	A key point that needs to be established in this approach is a regularity lemma; see Subsection~\ref{regularitysection}. It asserts that if a logarithmic form is \(\dbar\)-exact modulo an order-zero current supported on \(E\), then it can be written as the \(\dbar\) of a form with logarithmic poles. Many ideas in this section originate from the pioneering works \cite{Nog95} and \cite{LRW}.
	
	\medskip
	
	In the remaining sections (4, 5 and 6) we discuss several applications of the theory established above, including the extension of pluricanonical forms, the jumping of cohomology locus property for rank-one local systems on quasi-compact Kähler manifolds, and the deformation theory of log Calabi--Yau manifolds. Finally, we would like to mention that in a forthcoming paper \cite{CDHP2}, we generalise some of the above constructions to higher-rank local systems over quasi-projective varieties.
	
	For most of the results in this survey we will not provide a complete proof, but rather refer to the original articles. Notable exceptions are 
Theorem \ref{ddbar01} and Theorem \ref{logddbar}, for which we supply new arguments, mainly because we feel that they are slightly simpler than the original ones.
	\bigskip

\textbf{Acknowledgements:} The author would like to thank the collaborators on the results presented in this survey---Ya Deng, Christopher D.~Hacon, and Mihai Paun---for many fruitful discussions on conic Hodge theory and the applications presented here. In particular, the author is grateful to Mihai Paun for extensive discussions and collaborations on the all papers and the project described in this survey, as well as for his numerous suggestions for this survey, which greatly improved the manuscript.

The author acknowledges the support of the French Agence Nationale de la Recherche (ANR) under reference ANR-21-CE40-0010 and thanks the Institut Universitaire de France for providing excellent working conditions.

\tableofcontents

\section*{Notations and conventions}
\begin{itemize}
	\item Throughout this paper, \( X \) denotes a compact K\"ahler manifold, \( D =\sum D_i \) is a simple normal crossing divisor on \( X \).
	
	\item $\omega_\cC$ is a K\"ahler metric on $X \setminus D$ of conic singularity along $D$, cf. \eqref{conicmetric}.
	
	\item Let $L$ be a holomorphic line bundle on $X$ with a metric $h_L$ satisfying \eqref{singmetrin}. 
	We denote by $\cC^\infty(X,L)$  the space of $L$-valued forms smooth in conic sense with respect to $(\omega_\cC, h_L)$, cf. Definition \ref{forms}.
	
	\item We denote by $\Omega^p _X(\log D)$ the sheaf of holomorphic forms with logarithmic poles along $D$. 
	
	\item We denote by $A^{p, q} (X, D, L)$ the space of smooth $L$-valued forms with logarithmic poles along $D$. 
	
	\item Let $D=\Delta_1 +\Delta_2$ be a simple normal crossing divisor on \( X \). In addition to \eqref{singmetrin}, we assume that the equality
		\begin{equation}
		i\Theta_{h_L}(L)=\sum_{D_i\subset \Delta_2} a_i [D_i]+\theta
	\end{equation}
holds, i.e. we only allow the metric $h_L$ to be singular along the components of $\Delta_2$.	
	We denote by $\cC^\infty (X, \Delta_1, L)$ the space of $L$-valued forms with logarithmic poles along $\Delta_1$ which are smooth in conic sense along $\Delta_2$, cf. Definition \ref{logconic}.
	
	\item A Hermitian metric on \( L|_{X\setminus D} \) is said to be \emph{harmonic} if it satisfies the conditions described in \ref{harmonicmetric}. Let \( D_K'' \), \( D_K' \), and \( D_K \) be the associated connections on \( L \) as defined in section \ref{logflatsubsection}.

	\item \( M_{\rm DR}(X/D) \) denotes the moduli space of pairs \( (L, \nabla) \), where \( L \) is a holomorphic line bundle on \( X \), and \( \nabla: L \to L \otimes \Omega_X(\log D) \) is a logarithmic integrable connection. \( M_{\rm B}(X/D) \) denotes the moduli space of rank one local system on $X\setminus D$.
	
	\item To each pair \( (L, \nabla)  \in M_{\rm DR}(X/D) \) we can associate a natural decomposition $$D=\Delta_1 +\Delta_2 +\Delta_3$$ of 
	the divisor $D$, cf. \eqref{decomp} (each of the components $\Delta_i$ have different geometric interpretations).
	
	\item The metric $\omega_D$ we construct in this setup will have mixed singularities: of Poincaré-type on $\Delta_2$ and of
	conic type on $\Delta_3$, cf. \eqref{m1}.
	
	\item  In the same setting, we denote by 
	\(\cC^\infty_{K}(X,L)\) and \(\cC^\infty_{K}(X,\Delta_1,L)\) for the space of "smooth forms" and "logarithmic smooth forms" on \(X\) introduced in Definition \ref{smoothK} and Defintion \ref{logKsmooth}. 
	
	\item $\mathcal{D} (X,L)$ is dual space of $\cC^\infty (X, L^\star)$ and \(\mathcal{D}_{K}(X,L)\) is the dual space of 	\(\cC^\infty_{K}(X,L^\star)\), cf. Definition \ref{currents} and Definition \ref{currentsK}.  

\end{itemize}

%%%%%%%%%%%%%%%%%%%%%%%%%%%%%%%%%%%%%%%%%%%%%%%%
%%%%%%%%%%%%%%%%%%%%%%%%%%%%%%%%%%%%%%%%%%%%%%%%
	
\section{Hodge theory for $L$-valued forms}\label{conicsection}

Let $(X,\omega_X)$ be a compact K\"ahler manifold, and let $L\to X$ be a holomorphic line bundle. Assume that $L$ admits a possibly singular metric $h_L$ such that
\begin{equation}\label{singmetr}
	i\Theta_{h_L}(L)=\sum a_i [D_i]+\theta,
\end{equation}
where $D:=\sum D_i$ is an snc divisor, $a_i\in \mathbb Q$, and $\theta$ is a smooth form on $X$.
In this section, we present a Hodge theory for $(X,L,h_L)$, following the constructions in \cite{JCMP}.

To begin with, we fix an integer $m\in \mathbb N$ which is sufficiently large and divisible; in particular we assume that $m a_i\in \mathbb Z$ holds for every $i$.
We will systematically use the $m$-ramified covers corresponding to the orbifold structure
\[
\Delta := \sum_i \Bigl(1-\frac{1}{m}\Bigr)D_i.
\]
More precisely, let $\{(V_i,z_i)\}_{i\in I}$ be a finite cover of $X$ by coordinate charts such that
\begin{equation}\label{deck2}
	z_i^1\cdots z_i^r=0
\end{equation}
is a local equation of the divisor $D$ on $V_i$.
We then consider the local $m$-ramified maps
\begin{equation}\label{deck3}
	\pi_i:U_i\to V_i,\qquad
	\pi_i(w_i^1,\dots,w_i^n):=
	\bigl((w_i^1)^m,\dots,(w_i^r)^m,w_i^{r+1},\dots,w_i^n\bigr),
\end{equation}
which define the orbifold structure associated with $(X,\Delta)$.
Let $\omega_{\cC}$ be a conic K\"ahler metric with respect to $(X,\Delta)$; namely, $\omega_{\cC}$ is a smooth K\"ahler metric on $X\setminus D$, and $\pi_i^\star \omega_{\cC}$ extends to a smooth K\"ahler metric on $U_i$ for every $i$.
Such a conic metric can be constructed as follows:
\begin{equation}\label{conicmetric}
	\omega_{\cC}:= C\,\omega_X+\sum_i \sqrt{-1}\,\ddbar |s_{D_i}|_{h_i}^{2/m},
\end{equation}
where $s_{D_i}$ is the canonical section of $\O_X(D_i)$ and $h_i$ is a smooth metric on $\O_X(D_i)$.
If the positive constant $C$ is sufficiently large, then $\omega_{\cC}$ is a conic K\"ahler metric.

\medskip

\noindent In the absence of the twisting bundle $L$, the meaningful objects on $(X,\Delta)$ are the so-called orbifold differential forms, cf. for exmaple \cite{Rai96}.
The corresponding notion in our setting here is as follows.

\begin{defn}\label{forms}
	Let $\phi$ be an $L$-valued smooth form on $X\setminus D$. We say that $\phi$ is \emph{smooth in the conic sense} if the local quotients
	\begin{equation}\label{deck4}
		\phi_i := \frac{1}{w_i^{am}}\,\pi_i^\star(\phi|_{V_i})
	\end{equation}
	extend smoothly to $U_i$.
	Here we use the shorthand
	\[
	w_i^{am}:=(w_i^1)^{a_1 m}\cdots (w_i^r)^{a_r m}.
	\]
	The space of $L$-valued forms smooth in the conic sense is denoted by $\cC^\infty(X,L)$. Note that the defintion  depends on the metrics $h_L$ and $\omega_\cC$.
\end{defn}

Notice that this agrees with the usual notion of \emph{orbifold differential forms} in case $(L, h_L)$ is trivial. The idea here is similar, except that we also take into account the singularities of the metric $h_L$ on $L$.

The following statement relates the intrinsic differential operators associated with $(X,\omega_{\cC})$ and $(L,h_L)$ to the corresponding local operators, defined via the local uniformisations $\pi_i$.

\begin{proposition}\label{orbderiv}\cite[Prop 2.17]{JCMP}
	Let $\phi$ be an $L$-valued $(p,q)$-form which is smooth in the orbifold sense.
	Then its ``natural'' derivatives $D'_{h_L}$, $D_{h_L}^\star$, etc.\ are also smooth in the orbifold sense. Moreover:
	\begin{enumerate}
		\smallskip
		\item[\rm (1)] $\sup_{X\setminus Y} |\phi|_{h_L,\omega_{\cC}}<\infty$; in particular, forms which are smooth in the conic sense are bounded.
		\smallskip
		\item[\rm (2)] The following identities hold:
		\[
		\pi_i^\star(D'_{h_L}\phi)= w_i^{am} D'\phi_i, \qquad
		\pi_i^\star(\dbar \phi)= w_i^{am} \dbar \phi_i,
		\]
		and
		\[
		\pi_i^\star\bigl({D'}_{h_L}^{\star}\phi\bigr)= w_i^{am} {D'}^\star \phi_i,\qquad
		\pi_i^\star(\dbar^\star \phi)= w_i^{am} \dbar^\star \phi_i,
		\]
		where $D'\psi_i:=\partial \psi_i - \partial(\varphi_{L,0}\circ \pi_i)\wedge \psi_i$ and
		$\varphi_L=\sum a_i \log|z_i|^2 + \varphi_{L,0}$.
		\item[\rm (3)] Let $\Delta'' := [\dbar,\dbar^\star]$ be the Laplacian with respect to $(\omega_{\cC},h_L)$. Then
		\[
		\pi_i^\star(\Delta'' \phi)= w_i^{am}\cdot \Delta''_{\mathrm{sm}}\,\phi_i,
		\]
		where $\Delta''_{\mathrm{sm}}$ is the Laplacian in the local nonsingular setting $(\pi_i^\star \omega_{\cC},\, \varphi_{L,0}\circ \pi_i)$.
	\end{enumerate}
\end{proposition}

Note that, by our definition, given two forms $u$ and $v$ which are smooth in the conic sense,  
the equality
\[
\dbar u = v \qquad\text{ on } X
\]
holds in the conic sense means simply that $\dbar u = v$ on $X \setminus D$.
The following example/remark explains clearly an important difference between the usual $\dbar$ operator, and the one we introduce in the conic setting. 

\begin{exemple}

	Let
	\[
	A^{p,q}(X,D,L):= C^\infty_{(0,q)}\bigl(X,\Omega_X^p(\log D)\otimes L\bigr)
	\]
	be the space of $L$-valued smooth forms with logarithmic poles along $D$ in the classical sense, and let $s\in A^{p,q}(X,D,L)$.
	Assume that $a_i<0$ for every $i$ in \eqref{singmetr}. Then $s$ is smooth in the conic sense.
	
	For instance, suppose that $D_1$ is locally defined by $z_1=0$. Then $\displaystyle \frac{dz_1}{z_1}\big|_{X\setminus D}$ is smooth. Moreover, after the $m$-ramified cover $\pi$, we have
	\[
	w_1^{-a_1 m}\,\pi^\star\!\left(\frac{dz_1}{z_1}\right)
	= m\, w_1^{-a_1 m}\,\frac{dw_1}{w_1},
	\]
	and the right-hand side is smooth when $a_1<0$ (recall that $a_1 m\in \mathbb Z$). Hence $\displaystyle \frac{dz_1}{z_1}$ is smooth in conic sense.
	
	Moreover, we have $\displaystyle \dbar(\frac{dz_1}{z_1})=0$ in conic sense: indeed, the condition is that
	\[
	w_1\,\dbar\!\left(\frac{dw_1}{w_1}\Big|_{\DD^\star}\right) .
	\]
	extends across the disc, which is clearly the case. 
	 Notice that this is a bit confusing, since the form $\displaystyle \frac{dz_1}{z_1}$ is not $\dbar$-closed in the usual sense. Indeed, the equality
$$\dbar\!\left(\frac{dz_1}{z_1}\right) = -[z_1=0]$$
holds -up to some multiple.
\end{exemple}

\medskip

\begin{remark}
Let $u$ and $v$ be two $L$-valued forms, which are smooth in conic sense, and such that the equality $\dbar u= v$ holds pointwise in the complement of $D$. Then it is true that we have $\dbar u= v$ in conic sense, but as the previous example shows, this is not true in the usual sense.
\end{remark}

\medskip

We now list some basic properties of forms smooth in the conic sense. The following proposition shows that conic-smooth forms behave well with respect to integration by parts.

\begin{proposition}\cite[Cor 2.19]{JCMP}\label{int}
	Let $(L,h_L)$ be a holomorphic line bundle satisfying \eqref{singmetr}, and let $(L^\star,h_L^\star)$ be its dual endowed with the dual metric.
	Let $\alpha$ be an $L$-valued $(p,q)$-form and let $\beta$ be an $L^\star$-valued $(n-p-1,n-q)$-form. Assume that both $\alpha$ and $\beta$ are smooth in the conic sense.
	Then the usual integration by parts formula holds:
	\begin{equation}\label{van40}
		\int_X D'_{h_L}\alpha\wedge \beta
		= (-1)^{p+q+1}\int_X \alpha\wedge D'_{h_L^\star}\beta,
	\end{equation}
	where $D'_{h_L}$ and $D'_{h_L^\star}$ are understood in the conic sense.
\end{proposition}

For the proof of Proposition \ref{int}, we consider  the standard cut-off functions $\mu_\ep$ for $X\setminus D$, namely
\begin{equation}\label{cutoff}
\mu_\ep = \rho_\ep\!\left(\log\log\frac{1}{|s_D|_h^2}\right),
\end{equation}
with $s_D$ the canonical section of $\O_X(D)$, $h$ a smooth metric, and $\rho_\ep$ a function equal to $1$ on $[1,\ep^{-1}]$ and equal to $0$ on $[1+\ep^{-1},+\infty[$, with $|\rho_\ep'|\le 2$.  As the smooth forms in conic sense are smooth on $X\setminus D$ in the classical sense, we have 
$$	\int_X D'_{h_L} (\mu_\ep \cdot \alpha)\wedge \beta
= (-1)^{p+q+1}\int_X \mu_\ep \cdot \alpha\wedge D'_{h_L^\star}\beta .$$
Together with the fact that $\int_X |\partial \mu_\ep|_{\omega_\cC} \omega_\cC ^n < +\infty$, by letting $\ep \to 0$, we obtain \eqref{van40}.

\medskip

The standard G{\aa}rding and Sobolev inequalities, together with the Rellich embedding theorem, still hold in the conic setting; they can be deduced from their classical counterparts via local uniformizations of $(X,\Delta)$.
As a consequence, the Hodge decomposition also holds in the conic setting.

\begin{thm}\label{CoHo1}\cite[Thm 2.28]{JCMP}
	Let $(L,h_L)\to X$ be a line bundle endowed with a metric $h_L$ satisfying \eqref{singmetr}, and let $\omega_{\cC}$ be a K\"ahler metric with conic singularities as in \eqref{conicmetric}.
	Then we have the Hodge decompositions
	\begin{equation}\label{pave29111}
		L^2_{p,q}(X,L)=\Ker \Delta''_{h_L}\oplus \Imm \dbar \oplus \Imm \dbar^\star
	\end{equation}
	and
	\begin{equation}\label{pave291}
		\cC^\infty_{p,q}\bigl(X,(L,h_L)\bigr)
		= \Ker \Delta''_{h_L}\oplus \Imm \Delta''_{h_L}.
	\end{equation}
	Here all differential operators are understood in the conic sense.
\end{thm}

\medskip

Recall that one important method for solving the $\dbar$-equation and studying the deformation space is the $\ddbar$-lemma. The key point is that, unlike the case of $L^2$ estimates, we do not need strict positivity in order to solve the $\dbar$-equation. See Sections~\ref{extenpluri}, ~\ref{jumping}  and~\ref{logCY}  for details on this idea. 
\smallskip
 
As a consequence of Theorem~\ref{CoHo1}, we have the following conic version of the $\ddbar$-lemma.

\begin{cor}\label{conicddabr}
	Let $(L,h_L)\to X$ be a holomorphic line bundle endowed with a metric $h_L$ satisfying \eqref{singmetr}, and let $u$ be an $L$-valued $(p,q)$-form which is smooth in conic sense.
	\begin{itemize}
		\item If $\theta=0$ and $u\in \ker \dbar \cap \im D'_{h_L}$, then the equation
		\begin{equation}\label{ddbarconic}
		u=\dbar\circ D'_{h_L} v
		\end{equation}
		admits a solution $v$ which is moreover smooth in conic sense. 
		\item In general, if $\theta\ge 0$, let $u$ be an $L$-valued $(n,q)$-form smooth in conic sense, where $n=\dim X$.
		If $\dbar u=0$ in the conic sense and
		\[
		u = D'_{h_L} v_1 + \theta\wedge v_2
		\]
		in conic sense for some $v_1,v_2\in \cC^\infty(X,L)$, then $u=\dbar v$ in the conic sense for some $v\in \cC^\infty(X,L)$ which is orthogonal to the harmonic part $\ker \Delta''_{h_L}$.
	\end{itemize}
\end{cor}

The first statement follows directly from Theorem~\ref{CoHo1} and the proof of the classical $\ddbar$-lemma. The second follows from Theorem~\ref{CoHo1} and the Bochner identity in the conic sense; see \cite[Prop.~2.21]{JCMP}. 

\begin{remark} In \eqref{ddbarconic}, the equality only holds in conic sense, and as we have seen, this is not necessarily equivalent to the classical sense. It is therefore important to know when $u$ is $\ddbar$-exact in the usual sense. We refer to Theorem~\ref{classsol} for a result in this direction.
	\end{remark}

\medskip

Motivated by the applications we will present at the end of this survey, we next introduce conic-smooth forms with logarithmic poles.

\begin{defn}\label{logconic}
	Let $(X,\omega_X)$ be a compact K\"ahler manifold and let $D=\Delta_1+\Delta_2$ be an snc divisor on $X$.
	Let $L\to X$ be a holomorphic line bundle. Assume that $L$ admits a possibly singular metric $h_L$ such that
	\[
	i\Theta_{h_L}(L)=\sum_{D_i\subset \Delta_2} a_i [D_i]+\theta,
	\]
	where $a_i\in \mathbb Q$ and $\theta$ is a smooth form on $X$.
	Let $\omega_{\cC}$ be a conic K\"ahler metric with respect to the conic structure $(1-\frac{1}{m})\Delta_2$, and let $\pi_i:U_i\to V_i$ be the $m$-ramified cover along $\Delta_2$.
	
	Let $\phi$ be an $L$-valued smooth form on $X\setminus D$. We say that $\phi$ has \emph{logarithmic poles along $\Delta_1$ and is smooth in the conic sense} if
	\[
	\frac{\pi_i^\star \phi}{w_i^{am}}
	\]
	extends to a smooth form on $U_i$ with at most logarithmic poles along the inverse image of $\Delta_1$.
	We denote by $\cC^\infty(X,\Delta_1,L)$ the space of $L$-valued smooth conic forms with logarithmic poles along $\Delta_1$.
\end{defn}

We can establish a version of Hodge theory for $\cC^\infty(X,\Delta_1,L)$ as follows.
On the logarithmic tangent bundle $T_X\langle \Delta_1\rangle = (\Omega^1 _X ( \log \Delta_1))^\star$, we consider the Hermitian metric
\begin{equation}\label{log2}
	g_D:= \omega_{\cC} + \sqrt{-1}\sum_{D_i\subset \Delta_1}
	\frac{D's_{D_i}\wedge \ol{D's_{D_i}}}{|s_{D_i}|^2} ,
\end{equation}
where $s_{D_i}$ is the canonical section of $\mathcal{O}_X (D_i)$.
Then $g_D$ has conic singularities along $\Delta_2$ and we can use the local generators $(\frac{ds_{D_i}}{s_{D_i}})^\star$ for $D_i\subset \Delta_1$ in order to define an almost orthonormal basis on $X\setminus \Delta_2$.
Notice that $g_D$ (twisted with the metric $h_L$ of $L$) can be interpreted as \emph{singular Hermitian metric} on the bundle
\[
E_p:=\Omega_X^p(\log \Delta_1)\otimes L.
\]

Let $\pi_i: U_i \to V_i$ be the local $m$-ramified cover in \eqref{deck3}. 
For an $E_p$-valued $(0,q)$-form $\phi$,  $\pi_i ^* \phi$ can be seen as $\pi^\star _i L$-valued forms on $U_i \setminus (\pi_i)^{-1} (D)$. 
We say that $\phi$ is smooth in the conic sense if the $\pi^\star _i L$-valuded form
\[
\frac{\pi_i^\star \phi}{w_i^{am}} \qquad\text{ on } U_i \setminus (\pi_i)^{-1} (D)
\]
extends to a smooth form on $U_i$ with at most logarithmic poles along the inverse image of $\Delta_1$.
We denote by $\cC^\infty_{(0,q)}(X,E_p)$ the space of such $E_p$-valued $(0,q)$-forms.
Then we have a natural identification
\begin{equation}\label{identifi}
	\mathcal{I}:\cC^\infty_{(p,q)}(X,\Delta_1,L)\longrightarrow \cC^\infty_{(0,q)}(X,E_p),
\end{equation}
and we see next that the analogue of Theorem~\ref{CoHo1} holds for $(E_p,g_D)$ (where the ambient metric on $X$ is the same $\omega_{\cC}$).

\begin{thm}\label{CoHo2}
	We have the Hodge decomposition
	\[
	\cC^\infty_{(0,q)}(X,E_p)
	= \Ker \Delta''_{g_D,\omega_{\cC}} \oplus \Imm \Delta''_{g_D,\omega_{\cC}}.
	\]
	Here all differential operators are understood in the conic sense.
\end{thm}

\begin{proof}
	The arguments are virtually the same as in Theorem \ref{CoHo1}. Via the local $m$-ramified cover the inverse image 
	$$\frac{\pi_i^\star \phi}{w_i^{am}}$$
	of a form $\phi$ becomes a smooth form on $U_i$ with at most logarithmic poles along the inverse image of $\Delta_1$, namely a section of 
	$$\frac{\pi_i^\star \phi}{w_i^{am}} \in C^\infty (U_i, \Omega^p _{U_i} (\log \pi_i ^{-1}(\Delta_1)) \otimes L).$$ Note that $(\pi_i)^\star g_D$ is a smooth hermitian metric on $ \Omega^p _{U_i} (\log \pi_i ^{-1}(\Delta_1))$: this is the main reason for which we impose the conic singularities in the expression of $g_D$.  Then the standard G{\aa}rding, Sobolev inequalities, and the Rellich embedding theorem still hold in the this setting. From this point on, it is fairly easy to complete the details for the proof of Theorem \ref{CoHo2}.
	
	\end{proof}

Note that in general the metric $g_D$ is not flat. Therefore we do not have at our disposal a $\ddbar$-lemma for $\cC^\infty_{(0,q)}(X,E_p)$ even when $\theta=0$. 
This is one of the main motivation for establishing a version of Hodge decomposition theorem for currents, discussed in Section \ref{currentsect}.

\medskip

Given a form $u$ which is $\dbar$-exact in the conic sense, it is important to understand when $u$ is also $\dbar$-exact in the classical sense. 
We have the following result in this direction.

\begin{thm}\label{classsol}
	Let $X$ be a compact K\"ahler manifold and let $D=\Delta_1+\Delta_2$ be an snc divisor on $X$.
	Let $L\to X$ be a holomorphic line bundle endowed with a metric $h_L$ such that
	\[
	i\Theta_{h_L}(L)=\sum_{D_i\subset \Delta_2} a_i [D_i]+\theta,
	\]
	where $a_i\in [-1,0[\cap \mathbb Q$, $D=\sum D_i$ is snc, and $\theta$ is a smooth form.
	Let $A^{p,q}(X,D,L)$  be space of $L$-valued smooth form with (at most) logarithmic poles along $D$ in the classical sense.
	
	Let $u\in A^{p,q}(X,D,L)$. Since $a_i<0$, we have $u\in \cC^\infty_{(p,q)}(X,\Delta_1,L)$.
	Assume that $u$ is $\dbar$-exact in the conic sense, i.e.\
	\[
	u=\dbar v_1 \qquad \text{on } X\setminus \Delta_1
	\]
	in the conic sense for some $v_1\in \cC^\infty_{(p,q-1)}(X,\Delta_1,L)$.
	Then there exists $v\in A^{p,q-1}(X,D,L)$ such that
	\[
	u=\dbar v \qquad \text{on } X\setminus D.
	\]
\end{thm}

\begin{proof}
	The proof follows the ideas of \cite{LRW} and \cite{CDHP}.
	Consider the holomorphic vector bundle
	\[
	E:= \Omega_X^p(\log D)\otimes L
	\]
	endowed with a smooth Hermitian metric $h_E$, and let $(E^\star,h_E^\star)$ be its dual.
	We have a natural identification
	\[
	\mathcal{I}: A^{p,q}(X,D,L)\longrightarrow C^\infty_{(0,q)}(X,E),
	\]
	and a natural morphism
	\[
	\iota: C^\infty_{(n,n-q)}(X,E^\star)\longrightarrow
	C^\infty_{(0,n-q)}\bigl(X,\Omega_X^{n-p}\otimes L^\star\bigr)
	\]
	induced by contraction.
	
	Since $\dbar u=0$ on $X\setminus D$, we have $\dbar \mathcal{I}(u)=0$ on $X$.
	Fix a K\"ahler metric $\omega_X$ on $X$ and consider the smooth Hodge decomposition with respect to $(X,\omega_X)$ and $(E,h_E)$.
	Since $\mathcal{I}(u)$ is a $\dbar$-closed smooth $E$-valued form, we may write
	\[
	\mathcal{I}(u)=u_0+\dbar u_1,
	\]
	where $u_0$ is harmonic and $u_1\in C^\infty_{(0,q-1)}(X,E)$.
	Thus it suffices to show that $u_0=0$.
	
	\medskip
	
	Let $\sharp: C^\infty_{(0,q)}(X,E)\to C^\infty_{(n,n-q)}(X,E^\star)$ be the Hodge star operator.
	Since $u_0$ is harmonic, $\sharp u_0$ is also harmonic, cf. \cite[Chapter VI, Thm 7.3]{bookJP}. Then
	\[
	\|u_0\|_{L^2}^2
	= \int_X u_0\wedge \sharp u_0
	= \int_X \mathcal{I}(u)\wedge \sharp u_0 = \int_X  u\wedge \iota(\sharp u_0)
	= \lim_{\ep\to 0}\int_X \mu_\ep\, u\wedge \iota(\sharp u_0)
	= -\lim_{\ep\to 0}\int_X \dbar \mu_\ep \wedge v_1 \wedge \iota(\sharp u_0),
	\]
	where $\mu_\ep$ is the standard cut-off functions for $X\setminus D$ defined in \eqref{cutoff}.
	
	We suppose that $z_i =0$ defines locally $D_i$ and  $w_i=0$ defines the inverse image of $D_i$.  
	As $v_1\in \cC^\infty_{(p,q-1)}(X,\Delta_1,L)$ and $-1 \leq a_i <0$, 
then near a generic point of $w_i=0$ the form
	$d\ol w_i \wedge \pi^\star v_1$ may have a pole of the form 
	$$\frac{d\ol w_i\wedge dw_i}{w_i} ,$$
	but without higher order poles: if not, as $\pi^\star v_1$ is $\pi$-invariant, $d\ol w_i \wedge \pi^\star v_1$ should have a pole of the form 
	$\frac{d\ol w_i\wedge dw_i}{w_i ^{1+ m}}$. This contradicts with the fact that $v_1$ is smooth in the conic sense, as the coefficience $a_i \geq -1$.  On the other hand, because of the contraction in $\iota$, the wedge product $\iota(\sharp u_0) \wedge d z_i$ is divisible by $z_i$. As a consequence, the form
	\[
	d\ol w_i \wedge \pi^\star v_1 \wedge \pi^\star \big( \iota(\sharp u_0)\big)
	\]
	is smooth.  Consequently, by Lebesgue dominated convergence theorem, we have
	\[
	\lim_{\ep\to 0}\int_X \dbar \mu_\ep \wedge v_1 \wedge \iota(\sharp u_0)=0.
	\]
	Hence $\|u_0\|_{L^2}=0$, so $u_0=0$, which concludes the proof.
\end{proof}

To end this section, we ask the following question.

\begin{quest}
	Can one establish a strong Hodge decomposition in the setting where $h_L$ has analytic singularities, or even when the weight of $h_L$ is the difference of two quasi-psh functions?
	\medskip
	
In the case where $h_L$ has analytic singularities, after a modification of $X$, we may assume that the singularities of the metric $h_L$ are divisorial. However, when the Lelong numbers of $h_L$ are not rational, it is unclear what the metric on $X$ should be.
\end{quest}

%%%%%%%%%%%%%%%%%%%%%%%%%%%%%%%%%%%%%%%%%%%%%%%%%%%%%%%%%
%%%%%%%%%%%%%%%%%%%%%%%%%%%%%%%%%%%%%%%%%%%%%%%%%%%%%%%%
%%%%%%%%%%%%%%%%%%%%%%%%%%%%%%%%%%%%%%%%%%%%%%%%%%%%%%%%%

\section{Hodge theory for logarithmic holomorphic flat line bundles}\label{logflat}

Let $(X, D)$ be a pair consisting of a compact K\"ahler manifold $X$ and an snc divisor $D$. We denote by $M_{\rm B}(X_0)$ the space of rank one local systems on the open manifold $X\setminus D$.
Then the following beautiful result is due to Budur-Wang.
\begin{thm}[\cite{BW20}]\label{introcomp}
	Consider the germ of a holomorphic disk $\gamma: (\CC, 0)\to M_{\rm B} (X_0)$, and fix $\gamma(0)= \tau$. Let $k$ be a positive integer so that we have $\displaystyle \dim  H^p (X\setminus D, \gamma ( t)) \geq k$ for every $|t| \ll 1$.  Then we have
	\begin{equation} 
		\dim H^p (X\setminus D, \tau + t\dot\alpha )\geq k
	\end{equation}
	for $ |t|\ll 1$, where $\dot\alpha$ is the infinitesimal deformation  $d\gamma(0)$. 
\end{thm}

\noindent In the statement above, we denote by \(H^\bullet(X\setminus D, \tau)\) the cohomology groups with coefficients in \(\tau\).  
One of the main objectives in the recent article \cite{CDHP} is to propose a new proof of this result,  
by developing a version of the Hodge theory for rank-one local systems on quasi-compact K\"ahler manifolds. 
The proof of \ref{introcomp} is self-contained and differs substantially from the original argument in \cite{BW20}, which relies on global considerations involving moduli spaces and the Riemann–Hilbert correspondence (with some ideas tracing back to \cite{Sim93}).  
We now briefly outline the main ideas in \cite{CDHP}, a few key statements/facts will be presented more in detail in the next subsections.

Let $M_{\rm DR}(X/D)$ be the space of logarithmic flat bundles, i.e. pairs $(L, \nabla)$ consisting of line bundles $L\to X$ together with a holomorphic connection   
\[\nabla: L\to L\otimes \Omega_X(\log D)\]
with log poles along $D$ such that $\nabla^2= 0$. We have a natural map
\begin{equation}\label{introsurjj}
\rho: M_{\rm DR}(X/D)\to M_{\rm B}(X\setminus D),
\end{equation}
which is surjective, by a well-known result due to Deligne, see \cite[II, 6.10]{Del70}. As shown in the same reference, the correspondence between the 
two spaces in \eqref{introsurjj} goes much further: if we consider the complex 
\begin{equation}\label{introderham}
	0 \to \mathcal{O} (L) \to\mathcal{O} (L) \otimes \Omega_X^{1}(\log D) \to  \mathcal{O} (L) \otimes \Omega_X^{2}(\log D) \to \cdots \mathcal{O} (L) \otimes \Omega_X^{n}(\log D) \to 0 
	\end{equation}
induced by the connection $\nabla$, then we have $$\mathbb H^p (X, L \otimes \Omega_X ^{\bullet }(\log D) )\simeq H^p (X_0, \tau),$$
provided that the residue of $(L, \nabla)$ on any component of $D$ is not a positive integer. Here $\tau$ is the image of $(L, \nabla)$ via the map $\rho$. 
\smallskip

The strategy adopted in \cite{CDHP} for the proof of Theorem \ref{introcomp} consists in developing the necessary tools in order to construct a resolution of  \eqref{introderham}. Afterwards ideas from Hodge theory are coming into the picture: in \emph{loc. cit.} one shows that the cohomology of the resulting complex is computed by the kernel of a Laplace-type operator corresponding to $\nabla + \dbar$. Finally, a version of $\ddbar$-lemma adapted to the current context is established, and it is one of the main technical tool in the proof.

In order to be able to select "the best" representative in a cohomology class, one first needs to endow the respective bundle/manifold ($L$ and $(X, D)$ in our context) with metrics which are 
geometrically significant. In case of the line bundle $L$, this will be a harmonic metric. Its definition, together with the natural decomposition of the divisor $D$ (induced by the component of $M_{\rm B}(X\setminus D)$ containing $\tau$) is surveyed in the first two subsections below. The said decomposition of $D$ will basically dictate the choice of the ambient metric $\omega_D$ on $(X, D)$ -- it will have mixed 
conic-Poincaré type singularities. Then everything is in place for stating and proving the Hodge decomposition theorem, see Section \ref{hodgeloc}. The resolution of the complex \eqref{introderham} is constructed in Section \ref{resolutionsection}, and -unsurprisingly- it involves differential forms with log poles, which are not necessarily $L^2$ with respect to the objects mentioned above. In particular, we cannot apply directly the aforementioned Hodge theory.
Hence the need for considering more general objects, namely $L$-valued currents. Their definition and relevance in our context is discussed in Section \ref{currentsect}.

\medskip

\subsection{Harmonic metrics and basic Hodge identities}\label{logflatsubsection}

To begin with, we recall the definition of harmonic metrics and the basic Hodge identities for $(L,\nabla)\in M_{DR}(X/D)$, as established in the foundational work \cite{Sim92}.

Set $D_K:=\dbar+\nabla$, where $\dbar$ defines the holomorphic structure of $L$. Then $D_K^2=0$ on $X\setminus D$.
Let $h$ be a metric on $L$, smooth on $X\setminus D$ and with possible analytic singularities along $D$.
Let $D'_h$ be the $(1,0)$-part of the Chern connection of $(L,h)$, and set
\[
\theta_0 := \frac{D'_h-\nabla}{2}.
\]
Then $\theta_0$ is a $(1,0)$-form with possible logarithmic poles along $D$.
By construction, the two connections $D'_h+\dbar$ and $\nabla+(\dbar-2\ol\theta_0)$ preserve the metric $h$.
We can therefore define $D''_K$ as follows.

\begin{defn}\label{def:connection}
	We set
	\begin{equation}\label{dbarop}
		D''_K
		:= \frac{\dbar+(\dbar-2\ol\theta_0)}{2}+\frac{\nabla-D'_h}{2}
		= \dbar + (\theta_0-\ol\theta_0),
	\end{equation}
	\begin{equation}\label{e1}
		D'_K := D'_h+\theta_0+\ol\theta_0.
	\end{equation}
	and 
		\begin{equation}
		D^c _K := D''_K -D'_K.
	\end{equation}
\end{defn}

Note that, by construction, the equality
\[
D''_K + D'_K = \dbar+\nabla = D_K
\]
holds. Moreover, it turns out that 
the crucial Hodge identities are still valid in this setup.

\begin{proposition}
	Let $\omega$ be a K\"ahler metric on $X\setminus D$, and let $(D'_K)^\star$ and $(D''_K)^\star$ be the formal adjoints of $D'_K$ and $D''_K$ with respect to $(h,\omega)$.
	Then
	\begin{equation}\label{kahleridentity}
		(D'_K)^\star = i[\Lambda_\omega,D''_K]
		\qquad\text{and}\qquad
		(D''_K)^\star = -i[\Lambda_\omega,D'_K].
	\end{equation}
\end{proposition}

Given a local system $(L,\nabla)$, in general one cannot construct a Hermitian metric $h$ such that $D'_h=\nabla$. The substitute is the so-called harmonic metric, an important notion which we now recall.

\begin{defn}[Harmonic metric]\label{def:harmonic}
	We say that $h$ is a \emph{harmonic metric} if $(D''_K)^2=0$ on $X\setminus D$.
\end{defn}

The existence of harmonic metrics for higher-rank semisimple local systems is a deep issue; we refer to \cite{Sim92} for the compact K\"ahler case and to \cite{Moc06} for the quasi-projective case.
In our rank-one case, however, the situation is much simpler, thanks to the following caracterisation.

\begin{proposition}\label{harmonicmetric}
	A metric $h$ is harmonic if and only if $i\Theta_h(L)=0$ on $X\setminus D$.
\end{proposition}
The proof of Proposition \ref{harmonicmetric} follows from a direct computation, cf. \cite[Prop 6.6]{CDHP}

\medskip

Before discussing the existence of harmonic metrics, we recall the definition of the residue of a rank-one local system $(L,\nabla)$.

\begin{defn}
	Let $(L,\nabla)\in M_{DR}(X/D)$, and let $e_L$ be a local holomorphic frame of $L$.
	Assume that $D_i$ is locally defined by $z_1=0$. By compactness of $D_i$, there exists a constant $a_i$ such that
	\[
	\nabla e_L = \frac{a_i\,dz_1}{z_1}\wedge e_L + C^\infty.
	\]
	We call $a_i$ the residue of $(L,\nabla)$ along $D_i$, denoted by $\Res_{D_i}(L,\nabla)$.
	Obviously, this definition is independent of the choice of the frame $e_L$.
\end{defn}

\begin{proposition}\label{firstchern}
	Let $(L,\nabla)\in M_{DR}(X/D)$. Then
	\[
	c_1(L) = -\sum_i \Rel\bigl(\Res_{D_i}(L,\nabla)\bigr)\,[D_i]\in H^2(X,\mathbb Q).
	\]
	In particular, there always exists a harmonic metric $h$ on $(L,\nabla)$.
\end{proposition}

\begin{proof}
	Fix a smooth metric $h_L$ on $L$. Then $\alpha:=D'_{h_L}-\nabla$ is a $(1,0)$-form with logarithmic poles along $D$.
	We compute
	\[
	\dbar \alpha
	= [\dbar,\alpha]
	= [\dbar, D'_{h_L}-\nabla]
	= i\Theta_{h_L}(L) + \sum_i \Res_{D_i}(L,\nabla)\,[D_i]
	=0 \in H^2(X,\mathbb C).
	\]
	Since $i\Theta_{h_L}(L)\in H^2(X,\mathbb R)$, taking real parts yields
	\[
	c_1(L) = -\sum_i \Rel\bigl(\Res_{D_i}(L,\nabla)\bigr)\,[D_i]\in H^2(X,\mathbb R).
	\]
	Thus we can find a metric $h$ on $L$ such that
	\[
	i\Theta_h(L) = -\sum_i \Rel\bigl(\Res_{D_i}(L,\nabla)\bigr)\,[D_i].
	\]
	In particular, $h$ is harmonic on $(L,\nabla)$.
\end{proof}

Note that the harmonic metric is not unique. Indeed, let $a_i\in \mathbb R$ satisfy
\[
\sum_i a_i\, c_1(D_i)=0\in H^2(X,\mathbb R).
\]
Then, by the familiar $\ddbar$-lemma (see however Theorem~\ref{ddbarcurrent}), there exists a function $\phi$ such that
\[
\ddc \phi = \sum_i a_i [D_i].
\]
Therefore, if $h$ is harmonic, then the metric $h\,e^{-\phi}$ is also harmonic.

\medskip

Given a harmonic metric $h$, we have the following important result.

\begin{proposition}\cite[Section 1]{Sim92}\label{hodgelap}
	Consider the Laplacians
	\[
	\Delta_K := [D_K, D_K^\star],\qquad
	\Delta'_K := [D'_K,(D'_K)^\star],\qquad
	\Delta''_K := [D''_K,(D''_K)^\star].
	\]
	If $h$ is harmonic, then
	\begin{equation}\label{comparelaplace}
		\Delta''_K = \Delta'_K = \frac{1}{2}\Delta_K
		\qquad \text{on } X\setminus D.
	\end{equation}
	Moreover, $D''_K$, $D'_K$, $(D''_K)^\star$, and $(D'_K)^\star$ commute with $\Delta''_K$ and $\Delta'_K$.
\end{proposition}

In the compact K\"ahler case, as $\Delta_K$ is a smooth elliptic operator, the Hodge decomposition theorem for $\Delta_K$ holds true \cite[Section 2]{Sim92}. Together with Proposition~\ref{hodgelap}, this implies the $D'_K\circ D''_K$-lemma ($= D_K \circ D^c _K$) \cite[Lemma 2.1]{Sim92}, which is crucial for many applications.
In the quasi-compact K\"ahler case, one would like to establish an analogous Hodge decomposition theorem for $\Delta_K$, and hence a $D_K \circ D^c _K$-lemma. This is the aim of the next subsection.

\subsection{Hodge theory for logarithmic holomorphic flat line bundles}\label{hodgeloc}

To establish a Hodge decomposition theorem in the quasi-compact, twisted Kähler case, we start by choosing a harmonic metric $h_L$ on $L$.
Corresponding to the connected component of $M_B(X\setminus D)$ containing the image $\rho(L,\nabla) \in M_{DR}(X/D)$ we can associate a decomposition of the divisor $D$ 
\begin{equation}\label{decomp}
	D = \Delta_1 + \Delta_2 + \Delta_3 ,
\end{equation}
such that the following hold:
\begin{itemize}
	\item $\Delta_2 \subset D$ is the union of components $D_i$ such that there exists a form
	$s \in H^0(X, \Omega_X^1(D))$ with $\Res_{D_i}(s) \neq 0$.
	
	\item $\Delta_1 \subset D \setminus \Delta_2$ consists of components $D_i$ such that
	$\Res_{D_i}(L,\nabla) = 0$.
	
	\item $\Delta_3 := D \setminus (\Delta_1 \cup \Delta_2)$.
\end{itemize}

We will not provide a detailed account about this important fact, but refer to \cite{CDHP}, Propositions 6.11 and 6.12. However, for the comfort of the reader, we list next a few basic properties. There exist a harmonic metric $h_L$ on $L$ such that the following are verified.

\begin{enumerate}
\smallskip

\item[\rm (1)] The equality 
\[
\Res_{Y}(L,\nabla) = -\nu_{Y}(h_L)
\]
holds for any component $Y$ of $\Delta_1 + \Delta_3$, and the common value above is a rational number.
\smallskip

\item[\rm (2)] For any component $W$ of $\Delta_2$, we have $\nu(h_L) \neq -\Res(L,\nabla)$ along $W$. 
\smallskip

\item[\rm (3)] The equality \[H^0(X, \Omega_X^1(\log D))= H^0(X, \Omega_X^1(\log \Delta_2))\] 
holds, and as a consequence, we have 
\[\Res_{Y}(L,\nabla) = \Res_{Y}(L',\nabla')\]
for any component $Y$ of $\Delta_1 + \Delta_3$, provided that $\rho(L,\nabla)$ and $\rho(L',\nabla')$ belong to the same connected component of  $M_{DR}(X/D)$.
\smallskip

\item[\rm (4)] We can assume that $\Rel  \Res_{D_i}(L,\nabla) \in [0,1[$ after replacing $(L,\nabla)$ by 
$(L \otimes \textstyle\prod \mathcal{I}_{D_i}^{a_i}, \nabla' )$, where $\nabla' $ is the natural flat connection induced by $\nabla$. This comes from the fact that 
the map \eqref{introsurjj} has the following property 
\[
	\rho(L,\nabla)
	=
	\rho\bigl(L \otimes \textstyle\prod \mathcal{I}_{D_i}^{a_i}, \nabla'\bigr).
\]
%and thus we have
%\begin{equation}\label{rel2}
%	\Res_{D_i}(L,\nabla)
%	=
%	\Res_{D_i}\bigl(L \otimes \textstyle\prod \mathcal{I}_{D_i}^{a_i}, \nabla'\bigr)
%	+ a_i.
%\end{equation}
%It follows that we can 
\end{enumerate}

\medskip

The above discussion can be summarised as follows.
\begin{lemme}\label{choicemetric} We can construct a harmonic metric $h_L$ on $L$ such that the next assertions hold.
	\begin{itemize}
		\item $\nu_{Y}(h_L) = -\Res_{Y}(L,\nabla) = 0$ for any component $Y$ of $\Delta_1$;
		\item $\nu_{Y}(h_L) \neq -\Res_{Y}(L,\nabla)$ and $\nu_{Y}(h_L) < 0$ for any component $Y$ of $\Delta_2$;
		\item $\nu_Y(h_L) = -\Res_Y(L,\nabla) \in ]-1,0[ \cap \mathbb Q$ for any component $Y$ of $\Delta_3$.
	\end{itemize}
\end{lemme}

\medskip

We now define a metric on the base manifold~$X$.
Let $\omega_D$ be a Kähler metric on $X \setminus (\Delta_2 + \Delta_3)$ with
Poincaré singularities along the support of~$\Delta_2$ and conic singularities
of angle $2\pi(1-1/k)$ along each component of~$\Delta_3$, where
$k \geq 3$ is a sufficiently large and divisible integer.
In particular, $k \cdot \Res_{\Delta_3}(L,\nabla) \in \mathbb{Z}$.

Such a metric can be constructed explicitly.
Write $\Delta_2 = \sum_{i=1}^m D_i$ and $\Delta_3 = \sum_{i=m+1}^N D_i$.
We define
\begin{equation}\label{m1}
	\omega_D := \omega_X
	+ \sum_{i=m+1}^N \sqrt{-1}\,\ddbar |s_i|^{2/k}
	- \sum_{i=1}^m \sqrt{-1}\,\ddbar \log \log \frac{1}{|s_i|^2},
\end{equation}
where $\omega_X$ is a fixed Kähler metric on $X$, chosen sufficiently positive so that
$\omega_D \geq \frac{1}{2}\omega_X$.
Here $s_i$ denotes the canonical section of $\mathcal{O}_X(D_i)$ with respect to a smooth metric.

\begin{remark}
At a first sight, the choice of singularities of $\omega_D$ may seem artificial/arbitrary. However, in view of the next results we will discuss, this seems to be the right choice.
\end{remark}

\medskip

We now apply the techniques developed in Section~1 on the quasi-compact manifold
$X \setminus \Delta_2$, with respect to the metric $(L,h_L,\omega_D)$.
In particular, all differential operators
$D_K$, $D_K^\star$, $\Delta_K$, etc.\ are defined
with respect to the conic structure $(1-\frac{1}{k})\Delta_3$.
\medskip

The main estimate obtained in~\cite{CDHP} is the following. 
For similar considerations (in a simplified context), we refer to the elegant article \cite{Biq97}.

\begin{thm}\label{keyinequ}\cite[Thm 7.9]{CDHP}
Let $h_L$ be the harmonic metric in Lemma~\ref{choicemetric}.  There exists a constant $C>0$ such that for any
	$L$-valued form $u$ with compact support in $X\setminus \Delta_2$,
	smooth in conic sense, the inequality
	\begin{equation}\label{mainineq}
		\|u\|_{L^2}^2
		+ \int_X \langle \Delta_K u, u \rangle \, dV_{\omega_D}
		\geq \frac{1}{C}
		\int_X \Bigl(
		|u|^2 |\theta_0|^2
		+ \langle [D_{h_L}, D_{h_L}^\star] u, u \rangle
		\Bigr)\, dV_{\omega_D}
	\end{equation}
	holds.
\end{thm}

\begin{proof}[Idea of the proof]
	We briefly explain the idea of the proof.
	Using the identity $D_K = D_{h_L} + 2\theta_0$, we obtain
	\begin{equation}\label{form40}
		[D_K, D_K^\star]
		= [D_{h_L}, D_{h_L}^\star]
		+ 4[\theta_0, \theta^\star_0]
		+ 2[D_{h_L}, \theta^\star _0]
		+ 2[\theta_0, D_{h_L}^\star].
	\end{equation}
	We can show that
	\[
	[D_{h_L}', \theta^\star _0] u
	= \theta^\star_0 \,\lrcorner\, \nabla^{1,0}_{\mathrm{cov}} u,
	\]
	where $\nabla^{1,0}_{\mathrm{cov}} u$ denotes the $(1,0)$-covariant derivative.
	Moreover, using \cite{Siu82}, we have
	\[
	\|\nabla^{1,0}_{\mathrm{cov}} u\|^2
	\simeq
	\|\dbar u\|^2 + \|\dbar^\star u\|^2.
	\]
	Combining these three ingredients yields \eqref{mainineq}.
\end{proof}

Now we explain several consequences of Theorem \ref{keyinequ}.

\begin{itemize}
	\item If $u \in L^2$ and $\Delta_K u \in L^2$, then it follows from \eqref{mainineq} that
	\[
	\int_X \log^2 |s_{D_i}| \, |u|^2 \, \omega_D^n < +\infty
	\]
	for any $D_i \subset \Delta_2$. Indeed, this is due to the fact that the residues of $\theta_0$ on 
	the components of $\Delta_2$ are non-zero. 
	
	\item More generally, we consider the weighted measure
	\[
	d\mu_a := \prod_{D_i \subset \Delta_2} \log^{2a_i}(|s_{D_i}|^2)\, dV_{\omega_D},
	\]
	where $a_i \in \mathbb N$.
	Then there exist positive constants $C_i$ such that the inequality
	\begin{align}
		C_0 \int_X |u|^2 \, d\mu_a
		+ \int_X \langle \Delta_K u, u \rangle \, d\mu_a
		\geq {} &
		C_1 \int_X |\theta_0|_{\omega_D}^2 |u|^2 \, d\mu_a  \cr
		& {} + C_2 \int_X \langle [D_{h_L}, D^\star_{h_L}] u, u \rangle \, d\mu_a
		\nonumber
	\end{align}
	holds.
	
	\item Using the two properties above, we can show that for a harmonic form $u$, namely $u \in L^2$ and $\Delta_K u = 0$ on $X \setminus \Delta_2$, we have
	\begin{equation}\label{devrgood}
		\int_X |\nabla^k_{\mathrm{cov}} u|^2 \, d\mu_a < +\infty
	\end{equation}
	for any $a_i \in \mathbb N$ and $k \in \mathbb N$.
	Here $\nabla_{\mathrm{cov}}$ denotes the covariant derivative with respect to $(\omega_D, h_L)$.
\end{itemize}

\medskip

Note that since $\omega_D$ is complete along $\Delta_2$, there are several ways to define smooth forms on a complete metric space.
Motivated by the definition of the Schwartz space on $\mathbb R^n$ and by \eqref{devrgood}, we introduce the following definition.

\begin{defn}\label{smoothK}\cite[Definition 8.9]{CDHP}
		In the setting of this subsection, we denote by $\cC^\infty_K(X, L)$ the space of $L$-valued forms $u$ which are smooth in conic sense on $X \setminus \Delta_2$ and satisfy
	\[
	\int_X |\nabla^k_{\mathrm{cov}} u|^2 \, d\mu_a < +\infty
	\]
	for all $k \in \mathbb N$ and for all multi-indices $a = (a_i)_{i \in I} \in \mathbb N^I$.
\end{defn}

Here is another important property of the space $\cC^\infty_K(X, L)$.
\begin{remark}\label{realsmooth}
	Let $\beta \in H^0\bigl(X, \Omega_X^1(\log D)\bigr)$. By the definition of $\Delta_2$, the form $\beta$ has only logarithmic poles along $\Delta_2$. Then, by Definition~\ref{smoothK}, if $f \in \mathcal{C}^\infty_K(X, L)$, we have
	\[
	\beta \wedge f \in \mathcal{C}^\infty_K(X, L).
	\]
	
\end{remark}

Thanks to Theorem \ref{keyinequ} and the key conditions in Lemma \ref{choicemetric}, we obtain an elliptic regularity result for the operator $\Delta_K$.
In particular, we have the following statement.

\begin{proposition}
	Let $g \in L^2$ and $f \in \mathrm{Dom}(\Delta_K)$ satisfy $\Delta_K f = g$.
	If $g \in \cC^\infty_K(X, L)$, then
	\[
	f \in \cC^\infty_K(X, L), \qquad
	D_K f \in \cC^\infty_K(X, L), \qquad
	D_K^\star f \in \cC^\infty_K(X, L).
	\]
\end{proposition}
\medskip

Finally, the following strong Hodge decomposition theorem is established.

\begin{thm}\label{smoothhodge}\cite[Thm 8.20]{CDHP}
We have the strong Hodge decomposition
	\[
	\cC^\infty_K(X, L)
	= \ker \Delta_K \oplus \im \Delta_K
	= \ker \Delta_K \oplus \im D_K \oplus \im D_K^\star
	\]
	in conic sense.
	Moreover, $\ker \Delta_K$ is finite-dimensional.
\end{thm}

\begin{proof}[Sketch of the proof]
	The main idea is as follows.
	Since the metric $\omega_D$ is complete along $\Delta_2$, $L^2$-Hodge theory \cite{bookJP} yields the decomposition
	\[
	L^2_\bullet(X, L)
	= \ker \Delta_K \oplus \overline{\im \Delta_K}
	\]
	in conic sense.
	Thus the key point is to show that $\overline{\im \Delta_K} = \im \Delta_K$.
	
	Assume by contradiction that $\overline{\im \Delta_K} \neq \im \Delta_K$.
	Then there exists a sequence $f_m$ such that $\|f_m\|_{H^1} = 1$,
	$\Delta_K f_m \to 0$, and $f_m \perp \ker \Delta_K$.
	Passing to a subsequence, we have $f_m \rightharpoonup f_0$ weakly in $H^1$,
	with $\Delta_K f_0 = 0$ and $f_0 \perp \ker \Delta_K$, hence $f_0 = 0$.
	
	On the other hand, on any compact subset $K \Subset X \setminus \Delta_2$,
	elliptic regularity implies that $f_m \to f_0$ strongly in $H^1$.
	Near $\Delta_2$, Theorem \ref{keyinequ} together with the relations in Lemma \ref{choicemetric} shows that
	the $L^2$-norm of $f_m$ is arbitrarily small.
	Consequently, $\|f_0\|_{L^2} \geq c$ for some constant $c > 0$, which is a contradiction.
	
	The finiteness of $\dim \ker \Delta_K$ follows from a similar argument.
\end{proof}

One of the main consequences of Theorem \ref{smoothhodge} is the following $\ddbar$-lemma.

\begin{cor}\label{ddbarK}
	Let $u \in \cC^\infty_K(X, L)$.
	If $u \in \im D_K \cap \ker D_K^c$ on $X \setminus \Delta_2$ in conic sense, then
	\[
	u = D_K \circ D_K^c v \qquad \text{on } X \setminus \Delta_2
	\]
	in conic sense, for some $v \in \cC^\infty_K(X, L)$.
\end{cor}

We conclude with the following question.

\begin{quest}
	Can the above Hodge theory be generalized to the higher-rank case?
	This question is already very interesting when $\dim X = 1$.
\end{quest}

\medskip

%%%%%%%%%%%%%%%%%%%%%%%%%%%%%%%%%%%%%%%%%%%%%%%%%%%%%%
%%%%%%%%%%%%%%%%%%%%%%%%%%%%%%%%%%%%%%%%%%%%%%%%%%%%%%%%%
%%%%%%%%%%%%%%%%%%%%%%%%%%%%%%%%%%%%%%%%%%%

\subsection{Resolution of the Hypercohomology}\label{resolutionsection}

Given a logarithmic holomorphic flat line bundle $(L,\nabla) \in M_{DR}(X/D)$, it induces two natural objects.

First, as we have already mentioned at the beginning of the current section, the logarithmic flat connection $\nabla$ induces the complex
\begin{equation}\label{derham1}
	0 \to \mathcal{O}_X(L)
	\to \mathcal{O}_X(L) \otimes \Omega_X^{1}(\log D)
	\to \mathcal{O}_X(L) \otimes \Omega_X^{2}(\log D)
	\to \cdots
	\to \mathcal{O}_X(L) \otimes \Omega_X^{n}(\log D)
	\to 0 .
\end{equation}
Let $\mathbb H^\bullet(X, L \otimes \Omega_X^{\bullet}(\log D))$ denote the hypercohomology of \eqref{derham1}.

Second, the parallel transport of $\nabla$ induces a natural rank-one local system on $X \setminus D$.
Let $M_{\mathrm B}(X \setminus D)$ be the space of rank-one local systems on $X \setminus D$.
The parallel transport induces a natural map
\begin{equation}\label{surjj}
	\rho : M_{DR}(X/D) \longrightarrow M_{\mathrm B}(X \setminus D).
\end{equation}

By \cite[II, 6.10]{Del70}, this map is surjective: given a flat bundle on $X \setminus D$, it can be extended to a holomorphic line bundle on $X$, and the flat connection may acquire simple poles along $D$.
Moreover, the kernel of $\rho$ is isomorphic to $\mathbb Z^m$.
Indeed, given a logarithmic flat bundle $(L,\nabla)$, for any $a_i \in \mathbb Z$ we can construct new logarithmic flat bundles
$L \otimes \prod \mathcal{I}_{D_i}^{a_i}$, and $\nabla$ induces a flat connection $\nabla'$ on
$L \otimes \prod \mathcal{I}_{D_i}^{a_i}$ with logarithmic poles along $D$.

At the level of cohomology, we have the following fundamental result due to Deligne.

\begin{thm}\cite[II, 6.10]{Del70}\label{Deli}
	Let $(L,\nabla) \in M_{DR}(X/D)$.
	If $\Res_{D_i}(L,\nabla) \notin \mathbb N^\star$ for every $D_i$, then for all $p$,
	\[
	\mathbb H^p\bigl(X, L \otimes \Omega_X^{\bullet}(\log D)\bigr)
	\simeq
	H^p\bigl(X \setminus D, \rho(L,\nabla)\bigr).
	\]
\end{thm}
\medskip

In \cite{CDHP}, we constructed a natural resolution of the complex \eqref{derham1} using the space of smooth forms $\cC^\infty_K(X,L)$ defined in Definition \ref{smoothK}.
Combined with the Hodge decomposition theorem \ref{smoothhodge}, this allows us to compute
$H^\bullet(X \setminus D, \rho(L,\nabla))$ using $\cC^\infty_K(X,L)$.
We briefly explain these constructions below.

\medskip

Let $\omega_D$ be the Poincaré–conic metric defined in \eqref{m1}, and let
$\cC^\infty_K(X,L)$ be the space defined in Definition \ref{smoothK} with respect to $(h_L,\omega_D)$.

\medskip

Note that $h_L$ and $\omega_D$ are smooth near a generic point of $\Delta_1$.
Thus, for $u \in \cC^\infty_K(X,L)$, the form $u$ is smooth in the classical sense near a generic point of $\Delta_1$.
To study the sheaf $\mathcal O_X(L) \otimes \Omega_X^\bullet(\log D)$, we introduce the following logarithmic conic smooth space, as an analogue of Definition \ref{logconic}.

\begin{defn}\label{logKsmooth}
	Let $\cC^\infty_K(X,\Delta_1,L)$ denote the space of forms locally of the form
	\[
	\sum_{I \subset J} \frac{dz_I}{z_I} \wedge f_I,
	\]
	where $f_I \in \cC^\infty_K(X,L)$ and $\prod_{i \in J} z_i = 0$ locally defines $\Delta_1$.
	
	We denote by $\cC^\infty_K(\Delta_1,L)$ the sheaf generated by the local sections of
	$\cC^\infty_K(X,\Delta_1,L)$.
\end{defn}

We then have the following local solvability result.

\begin{thm}\label{localsolution}\cite[Cor 9.9]{CDHP}
	Let $h_L$ be the harmonic in Lemma \ref{choicemetric}.
	Let $U \subset X$ be a small Stein open set, and let
	$f \in \cC^\infty_{s,K}(U,\Delta_1,L)$ satisfy $D_K f = 0$ and $f \in L^2(U)$.
	\begin{enumerate}
		\item If $s > n$, there exists $g \in \cC^\infty_{s-1,K}(U,\Delta_1,L)$ such that
		$g \in  L^2(U)$ and $D_K g = f$ on $U \setminus D$.
		
		\item If $s \leq n$, there exists $g \in \cC^\infty_{s-1,K}(U,\Delta_1,L)$ such that
		$g \in L^2(U)$ and
		\[
		f - D_K g \in H^0\bigl(U, \Omega_X^s(\log D) \otimes L\bigr),
		\qquad
		\nabla(f - D_K g) = 0.
		\]
	\end{enumerate}
\end{thm}

\begin{proof}
	The proof relies on an estimate due to A.~Fujiki~\cite{Fuj92} together with the elliptic estimate \eqref{mainineq}.
	We refer to \cite[Cor 9.9]{CDHP} for details.
\end{proof}

By the choice of the harmonic metric in Lemma \ref{choicemetric}, we have
$\nu(h_L) < 0$ on $\Delta_2 + \Delta_3$ and $\nu(h_L) = 0$ on $\Delta_1$.
Therefore, there is a natural embedding of sheaves
\begin{equation}\label{embed}
i: \Omega_X^\bullet(\log D) \otimes L
\subset
\cC^\infty_{\bullet,K}(\Delta_1,L).
\end{equation}

As a consequence of Theorem \ref{localsolution}, we obtain the following result.

\begin{thm}\label{quasiiso}
		The natural inclusion morphism $i$ \eqref{embed}  between the following complexes of sheaves is a quasi-isomorphism:
	\[
i:	(\Omega_X^\bullet(\log D) \otimes L, \nabla)
	\longrightarrow
 (	\cC^\infty_{\bullet,K}(\Delta_1,L), D_K)
	\]
	In particular, we have an isomorphism
\[ 
\frac{\ker D_K : \cC^\infty_{\bullet,K}(X, \Delta_1, L) \to \cC^\infty_{\bullet+1,K}(X, \Delta_1, L)}
{\im D_K : \cC^\infty_{\bullet-1,K}(X, \Delta_1, L) \to \cC^\infty_{\bullet,K}(X, \Delta_1, L)}
\simeq \mathbb{H}^\bullet (X, L \otimes \Omega_X^\bullet(\log D)).
\]
\end{thm}

In the special case $\Delta_1 = 0$, we obtain the following refinement.

\begin{thm}\label{L2cohoisom}
If $\Delta_1 = 0$, then the following three spaces are canonically isomorphic:
\[
\ker \Delta_K \simeq 
\frac{\ker D_K: \cC^\infty_{\bullet,K}(X, \Delta_1, L) \to \cC^\infty_{\bullet+1,K}(X, \Delta_1, L)}
{\im D_K: \cC^\infty_{\bullet-1,K}(X, \Delta_1, L) \to \cC^\infty_{\bullet,K}(X, \Delta_1, L)}
\simeq \mathbb{H}^\bullet (X, L \otimes \Omega_X^\bullet(\log D)).
\]
\end{thm}

\begin{proof}
	This follows directly from Theorem \ref{smoothhodge} and Theorem \ref{quasiiso}.
\end{proof}

%%%%%%%%%%%%%%%
%%%%%%%%%%%%%%%%%%%%%%%%%%%%
%%%%%%%%%%%%%%%%%%%%%%%%%%%%%%%%%%%%%%%%%%%%%%%%%
%%%%%%%%%%%%%%%%%%%%%%%%%%%%%%%%%%%%%%%%%%%%%%%

\section{Hodge decomposition for currents and a regularity lemma}\label{currentsect}

In the previous two sections, we established a Hodge theory for the conic smooth forms $\mathcal{C}^\infty(X,L)$ and $\mathcal{C}^\infty_K(X,L)$. In particular, we obtained the $\ddbar$-lemma (Theorems~\ref{conicddabr} and \ref{ddbarK}) for forms in $\mathcal{C}^\infty(X,L)$ and $\mathcal{C}^\infty_K(X,L)$. However, as we already observed in Theorem~\ref{quasiiso}, in order to study the hypercohomology
\[
\mathbb{H}^\bullet\bigl(X, L \otimes \Omega_X^\bullet(\log D)\bigr),
\]
we need to consider a larger space, namely $\mathcal{C}^\infty_K(X,\Delta_1,L)$.

Thanks to a nice idea due to Noguchi~\cite{Nog95}, one can study $\mathcal{C}^\infty_K(X,\Delta_1,L)$ by embedding it into the dual of $\mathcal{C}^\infty_K(X,L)$, i.e.\ into the space of currents. There are two key ingredients in this approach. The first is a Hodge decomposition for currents, namely the so-called de~Rham--Kodaira decomposition. The second is a regularity lemma, which says that if a logarithmic form is $\dbar$-exact modulo an order-zero current supported on $\Delta_1$, then one can find a logarithmic solution of the $\dbar$-equation. 
In this section we recall part of these ideas, and explain how they are implemented in our general setting.
\bigskip

\subsection{De~Rham--Kodaira decomposition and the $\ddbar$-lemma for currents}

We first briefly recall the classical de~Rham--Kodaira decomposition theorem~\cite{dRKo}.
Let $X$ be a compact K\"ahler manifold endowed with a K\"ahler metric $\omega_X$.
Let $T$ be a $(p,q)$-current on $X$, and let $(\xi_i)_i$ be an orthonormal basis of harmonic forms of type $(p,q)$.
We define the \emph{harmonic projection} of $T$ by
\begin{equation}\label{projH}
	\cH(T):=\sum_i \langle T,\xi_i\rangle\, \xi_i,
\end{equation}
where the pairing $\langle T,\xi_i\rangle$ is given by
\begin{equation}\label{pave72}
	\langle T,\xi\rangle := \int_X T\wedge \sharp \xi.
\end{equation}
Here $\sharp$ denotes the Hodge star operator. In particular, $\xi$ is harmonic if and only if $\sharp \xi$ is harmonic \cite[Chapter VI, (3.18 )]{bookJP}

\bigskip

We define, by duality, the Green operator on currents by
\begin{equation}\label{pave17}
	\int_X \cG(T)\wedge \phi := \int_X T\wedge G(\phi),
\end{equation}
for any smooth form $\phi$, where $G$ is the (smooth) Green operator. Then we have the de~Rham--Kodaira decomposition.

\begin{thm}[de~Rham--Kodaira]\label{decompcurrent}
	We have the decomposition
	\begin{equation}\label{pave18}
		T = \cH(T) + \Delta''\bigl(\cG(T)\bigr).
	\end{equation}
	If moreover $T$ is $\dbar$-closed, then
	\begin{equation}\label{pave19}
		T = \cH(T) + \dbar T_1,
	\end{equation}
	where $T_1 := \dbar^\star\bigl(\cG(T)\bigr)$.
\end{thm}

\begin{proof}
	The first statement follows directly from the smooth Hodge decomposition and the definitions of $\cH$ and $\cG$.
	
	For the second statement, note that $\cH(T)$ is $\dbar$-closed by \eqref{projH}. Applying $\dbar$ to \eqref{pave18} and using $\dbar T=0$, we obtain
	\[
	\Delta''\, \dbar\bigl(\cG(T)\bigr)=0.
	\]
	It follows that $\dbar\bigl(\cG(T)\bigr)$ is smooth, and then an integration-by-parts argument yields
	\[
	\dbar^\star \dbar\bigl(\cG(T)\bigr)=0.
	\]
	This implies \eqref{pave19}.
\end{proof}

As an application of Theorem~\ref{decompcurrent}, we obtain the $\ddbar$-lemma for currents.

\begin{thm}\label{ddbarcurrent}
	Let $T$ be a current which is $\partial$-closed and $\dbar$-exact (in the sense of currents). Then there exists a current $S$ such that
	\[
	T = \ddbar S.
	\]
	In particular, if moreover $T$ has total degree $1$, then $T=0$.
\end{thm}

As a nice application of Theorem~\ref{ddbarcurrent}, we have the following result due to Noguchi \cite{Nog95}. The projective case is proved in \cite{Del71}.

\begin{cor}\cite{Nog95}\label{Noglemma}
	Let $X$ be a compact K\"ahler manifold. Let $\alpha \in H^0\bigl(X,\Omega_X^1(\log D)\bigr)$. Then $\partial \alpha = 0$.
\end{cor}

\begin{proof}
	Since $\alpha$ is a holomorphic one form, 
	$\dbar \alpha$ is given by a collection of constant functions on the components of $D$. In particular, this current is $\partial$-closed
and obviously, $\dbar$-exact.
	
By the $\ddbar$-lemma (Theorem~\ref{ddbarcurrent}), there exists a current $T$ such that
	\[
	\dbar \alpha = \ddbar T.
	\]
	Then $\alpha - \partial T$ is a $\dbar$-closed $(1,0)$-current on $X$.
	By elliptic regularity for $\dbar+\dbar^\star$, the current $\alpha-\partial T$ is smooth, and hence
	\[
	\alpha-\partial T \in H^0\bigl(X,\Omega_X^1\bigr).
	\]
	In particular, $\alpha-\partial T$ is $\partial$-closed, and therefore so is $\alpha$.
\end{proof}

\begin{remark}
	More generally, as proved in \cite{Del71,Nog95}, for any $\alpha \in H^0\bigl(X,\Omega_X^k(\log D)\bigr)$ one has $\partial \alpha = 0$. The proof uses the above argument together with an induction on the residues of $\alpha$ on every stratum of $D$.
\end{remark}

\medskip

Using similar ideas, we can generalize Theorem~\ref{decompcurrent} and Theorem~\ref{ddbarcurrent} to the settings of Sections \ref{conicsection} and \ref{logflat}. We first define currents in conic sense.

\begin{defn}\label{currents}
	Let $(X,\omega_X)$ be a compact K\"ahler manifold and let $L\to X$ be a holomorphic line bundle. Assume that $L$ admits a (possibly singular) metric $h_L$ such that
	\begin{equation}
		i\Theta_{h_L}(L)=\sum a_i [D_i] + \theta,
	\end{equation}
	where $D:=\sum D_i$ is an snc divisor, $a_i\in \mathbb Q$, and $\theta$ is a smooth $(1,1)$-form on $X$.
	
	Let $\omega_{\cC}$ be a conic K\"ahler metric as in \eqref{conicmetric}.
	An $L$-valued conic current $T$ of $(p,q)$-type is an $L$-valued current such that there exist a constant $C>0$ and an integer $d\ge 0$ for which
	\begin{equation}\label{deck5}
		\left|\int_X T\wedge \phi\right|^2
		\le C \sum_{s=0}^d \sup_{X\setminus Y} \bigl|\nabla^s \phi\bigr|^2_{h_L^\star,\omega_{\cC}}
	\end{equation}
	for all $L^\star$-valued $(n-p,n-q)$-forms $\phi$ which are smooth in conic sense with respect to $h_L^\star$.
	
	We denote by $\mathcal{D}(X,L)$ the space of all $L$-valued conic currents (with respect to $h_L, \omega_\cC$).
\end{defn}

\begin{exemple}\label{imporantex}
	Let $\Delta_1$ be an snc divisor such that $\Delta_1 + D$ is snc and has no common components.
	Let $u \in \mathcal{C}^\infty(X,\Delta_1,L)$ be a form which is smooth in conic sense and has logarithmic poles along $\Delta_1$ (cf.\ Definition~\ref{logconic}). Then $u\in \mathcal{D}(X,L)$.
	
	Note that, when interpreted as conic current, the derivative $\dbar u$ may have a divisorial part supported on $\Delta_1$.
\end{exemple}

In analogy with the classical de~Rham--Kodaira theorem, we can define the Green operator $\cG$ and the harmonic projection $\cH$ for currents in $\mathcal{D}(X,L)$. By the same argument as above, we obtain the following the de~Rham--Kodaira decomposition and the $\ddbar$-lemma for conic currents.

\begin{thm}[de~Rham--Kodaira]\label{decompcurrentconic}
	In the setting of Definition~\ref{currents}, let $T\in \mathcal{D}(X,L)$. Then
	\[
	T = \cH(T) + \Delta''\bigl(\cG(T)\bigr).
	\]
	If moreover $T$ is $\dbar$-closed in conic sense, then
	\[
	T = \cH(T) + \dbar T_1,
	\qquad T_1 := \dbar^\star\bigl(\cG(T)\bigr).
	\]
\end{thm}

\begin{thm}\label{ddbarcurrentconic}
	In the setting of Definition~\ref{currents}, let $T\in \mathcal{D}(X,L)$.
	If $T$ is $\partial$-closed and $\dbar$-exact in conic current sense, then there exists $S\in \mathcal{D}(X,L)$ such that
	\[
	T = \ddbar S
	\]
	in conic current sense. In particular, if moreover $T$ has total degree $1$, then $T=0$.
\end{thm}

\medskip

Now let $(L,\nabla)\in M_{DR}(X/D)$. We can also define the current space, i.e.\ the dual space of $\cC^\infty_K(X,L)$, as follows.

\begin{defn}\label{currentsK}
	In the setting of Definition~\ref{smoothK}, the space $\mathcal{D}_K(X,L)$
	consists in $L$-valued currents $T$ such that there exist a constant $C>0$ and positive integers $d,a$ for which the inequality
	\begin{equation}\label{deck5K}
		\left|\int_X T\wedge \phi\right|^2
		\le C \sum_{s=0}^d \int_X \bigl|\nabla_{\mathrm{cov}}^s \phi\bigr|^2_{h_L^\star,\omega_{\cC}}\, d\mu_a
	\end{equation}
holds, for all $\phi \in \cC^\infty_K(X,L^\star)$.
\end{defn}

The differential $D_K$ induces an operator op the space of currents. Notice first that 
since $h_L$ is smooth at a generic point of $\Delta_1$,  we have the inclusion
	$$\cC^\infty_K(X,\Delta_1,L) \subset \mathcal{D}_K(X,L).$$
	
\begin{defn}\label{diffcurrents}	
We denote by $\mathcal{D}_K$ the operator on $\mathcal{D}_K(X,L)$ defined by duality 
	$$\int_X \mathcal{D}_K T \wedge \phi = (-1)^\bullet\int_XT \wedge D_K \phi $$
	for every $T\in \mathcal{D}_K(X,L)$ and $\phi \in \cC^\infty _K (X, L^\star)$. 
	Then $\mathcal{D}_K$ acts on $\mathcal{D}_K(X,L)$:
	\[
	\mathcal{D}_K:\mathcal{D}_K(X,L)\to \mathcal{D}_K(X,L).
	\]
\end{defn}

Notice that we can equally have the operator
	\[
	D_K:\cC^\infty_K(X,\Delta_1,L)\to \cC^\infty_K(X,\Delta_1,L)
	\]
by disregarding the current component that may occur.
Then $\mathcal{D}_K$ coincides with $D_K$ on $\cC^\infty_K(X,L)$, but not on $\cC^\infty_K(X,\Delta_1, L)$: $\mathcal{D}_K u$ may contain a divisorial part supported on $\Delta_1$ if $u$ has a logarithmic pole along $\Delta_1$.

As in conic case, we also have a de~Rham--Kodaira decomposition and a $\mathcal{D}_K\mathcal{D}_K^c$-lemma for currents in $\mathcal{D}_K(X,L)$.

\begin{thm}[de~Rham--Kodaira]\label{decompcurrentconicK}
	In the setting of Definition~\ref{currentsK}, let $T\in \mathcal{D}_K(X,L)$.
	Then
	\[
	T = \cH(T) + \Delta_K\bigl(\cG(T)\bigr).
	\]
	If moreover $T$ is $\mathcal{D}_K$-closed, then
	\[
	T = \cH(T) + \mathcal{D}_K T_1,
	\qquad T_1 := \mathcal{D}_K^\star\bigl(\cG(T)\bigr).
	\]
\end{thm}

\begin{thm}\label{ddbarcurrentconicK}
	In the setting of Definition~\ref{currentsK}, let $T\in \mathcal{D}_K(X,L)$.
	If $T$ is $\mathcal{D}_K$-closed and $\mathcal{D}_K^c$-exact, then there exists $S\in \mathcal{D}_K(X,L)$ such that
	\[
	T = \mathcal{D}_K\circ \mathcal{D}_K^c\, S.
	\]
	In particular, if moreover $T$ has total degree $1$, then $T=0$.
\end{thm}

			\bigskip
			\subsection{Regularity lemma}\label{regularitysection}
			
			In this subsection, we address another issue mentioned in the introduction to this section. Namely, given a logarithmic form $u$, if $u$ is $\dbar$-exact in the sense of currents modulo a divisorial part, then $u$ is $\dbar$-exact in the logarithmically smooth sense.
			This was proved in \cite{LRW} in the classical setting and was generalized to the current space $\mathcal{D}_K(X,L)$ in \cite{CDHP}. To explain the idea, we sketch here the proof in the conic current setting.
			
			\begin{lemme}\label{reglemma1}
				Let $X$ be a compact K\"ahler manifold, and let $D=\Delta_1+\Delta_2$ be an snc divisor on $X$.
				Let $L$ be a holomorphic line bundle on $X$ endowed with a possibly singular metric $h_L$ such that
				\[
				i\Theta_{h_L}(L)=\sum_{D_i\subset \Delta_2} a_i [D_i]+\theta
				\]
				for some $a_i\in \mathbb Q$, where $\theta$ is a smooth form.
				Let $u\in \cC^\infty_{(p,q)}(X,\Delta_1,L)$ be such that
				\begin{equation}\label{eqcurrent}
					u=\dbar T + T_{\Delta_1}
				\end{equation}
				in the conic current sense, where $T$ is a conic current and $T_{\Delta_1}$ is a order-zero conic current supported on $\Delta_1$.
				Then there exists $v\in \cC^\infty_{(p,q-1)}(X,\Delta_1,L)$ such that
				\[
				u=\dbar v \qquad \text{on } X\setminus \Delta_1
				\]
				in the conic sense.
			\end{lemme}
			
			\begin{proof}
				Set
				\[
				E := \Omega_X^p(\log \Delta_1)\otimes L,
				\]
				which is a holomorphic vector bundle. Let $g_D$ be the Hermitian metric on $E$ induced by \eqref{log2}.
				Let $\mathcal{C}^\infty_{(0,q)}(X,E)$ be the space introduced before Theorem~\ref{CoHo2}.
				As explained there, we have an identification
				\[
				\cI:\cC^\infty_{p,q}(X,\Delta_1,L)\longrightarrow \cC^\infty_{0,q}(X,E).
				\]
				Let $(E^\star,g_D^\star)$ be the dual bundle with the dual metric, which similarly gives rise to the space $\cC^\infty_{n,q}(X,E^\star)$.
				We also introduced the natural map
				\[
				\iota:\cC^\infty_{n,q}(X,E^\star)\longrightarrow \cC^\infty_{n-p,n-q}(X,L^\star),
				\]
				obtained by contraction with vector fields. One checks that
				\begin{equation}\label{contracteq}
					\int_X \cI(u)\wedge s \;=\; \int_X u\wedge \iota(s)
				\end{equation}
				for every $u\in \cC^\infty_{p,q}(X,\Delta_1,L)$ and $s\in \cC^\infty_{n,q}(X,E^\star)$.
				
				\medskip
				
				We remark that \eqref{eqcurrent} implies $\dbar \cI(u)=0$ in the conic sense.
				Therefore, the Hodge decomposition theorem (Theorem~\ref{CoHo2}) yields
				\begin{equation}\label{reg4}
					\cI(u)=\mathcal H + \dbar v_1,
				\end{equation}
				where $\mathcal H$ is $\Delta''$-harmonic and $v_1\in \cC^\infty_{(0,q-1)}(X,E)$.
				
				\medskip
				
				We claim that $\mathcal H=0$. Introduce the following antilinear operator
				\[
				\sharp:\cC^\infty_{(0,q)}(X,E)\longrightarrow \cC^\infty_{(n,n-q)}(X,E^\star)
				\]
				defined by the property that
				\[
				\int_X \langle s_1,s_2\rangle_{g_D}\,\omega_{\cC}^n
				= \int_X s_1\wedge \sharp s_2
				\]
				for all $s_1,s_2\in \cC^\infty_{(0,q)}(X,E)$.
				Since $\mathcal H$ is harmonic, we have
				\begin{equation}\label{harmpart}
					\dbar(\sharp \mathcal H)=0.
				\end{equation}
				Moreover,
				\begin{align}\label{reg5}
					\int_X |\mathcal H|^2\,\omega_{\cC}^n
					&= \int_X \langle \cI(u),\mathcal H\rangle\,\omega_{\cC}^n
					= \int_X \cI(u)\wedge \sharp \mathcal H \\
					&= \int_X u\wedge \iota(\sharp \mathcal H), \nonumber
				\end{align}
				where the last equality follows from \eqref{contracteq}. Using \eqref{eqcurrent}, we obtain
				\begin{equation}\label{reg7}
					\int_X u\wedge \iota(\sharp \mathcal H)
					= \int_X (\dbar T + T_{\Delta_1})\wedge \iota(\sharp \mathcal H).
				\end{equation}
				Note that $\iota(\sharp \mathcal H)$ is obtained by contraction with $T_X\langle \Delta_1\rangle$. Since $T_{\Delta_1}$ has order $0$, we have
				\[
				\int_X T_{\Delta_1}\wedge \iota(\sharp \mathcal H)=0.
				\]
				Therefore,
				\begin{align}\label{reg8}
					\int_X |\mathcal H|^2\,\omega_{\cC}^n
					&= \int_X \dbar T\wedge \iota(\sharp \mathcal H)
					= (-1)^\bullet \int_X T\wedge \dbar\bigl(\iota(\sharp \mathcal H)\bigr) \\
					&= (-1)^\bullet \int_X T\wedge \iota\bigl(\dbar(\sharp \mathcal H)\bigr)
					=0, \nonumber
				\end{align}
				where the last equality follows from \eqref{harmpart}. Hence $\mathcal H=0$.
				
				\medskip
				
				Consequently, \eqref{reg4} gives $\cI(u)=\dbar v_1$ for some $v_1\in \cC^\infty_{(0,q-1)}(X,E)$. Setting $v:=\cI^{-1}(v_1)\in \cC^\infty_{(p,q-1)}(X,\Delta_1,L)$, we obtain
				\[
				u=\dbar v \qquad \text{on } X\setminus \Delta_1
				\]
				in the conic sense.
			\end{proof}
			
			Now, combining Theorem~\ref{classsol} and Lemma~\ref{reglemma1}, we obtain the following corollary.
			
			\begin{cor}\label{regcor}
				Let $X$ be a compact K\"ahler manifold, and let $D=\Delta_1+\Delta_2$ be an snc divisor on $X$.
				Let $L$ be a holomorphic line bundle on $X$ endowed with a possibly singular metric $h_L$ such that
				\[
				i\Theta_{h_L}(L)=\sum_{D_i\subset \Delta_2} a_i [D_i]+\theta
				\]
				for some $a_i\in [-1,0[\cap \mathbb Q$, where $\theta$ is a smooth form.
				Let $A^{p,q}(X,D,L)$ be the space of $L$-valued smooth form with logarithmic poles along $D$ in the classical sense. 
				
				Let $u\in A^{p,q}(X,D,L)$. Then
				\[
				u\in \cC^\infty(X,\Delta_1,L)\subset \mathcal{D}(X,L).
				\]
				Moreover, if
				\[
				u=\dbar T + T_{\Delta_1}
				\]
				in the conic current space $\mathcal{D}(X,L)$, where $T\in \mathcal{D}(X,L)$ and $T_{\Delta_1}\in \mathcal{D}(X,L)$ is a order-zero current supported on $\Delta_1$, then there exists
				a $\beta \in A^{p,q-1}(X,D,L)$ such that 
				\[
				u=\dbar \beta \qquad \text{on } X\setminus D.
				\]
			\end{cor}
			
			Lemma~\ref{reglemma1} can be generalized to the space $\cC^\infty_K(X,\Delta_1,L)$ with respect to the operator $D_K$. More precisely, we proved the following regularity lemma.
			
			\begin{lemme}\label{regularprop}\cite[Prop 10.17]{CDHP}
				In the setting of Definition~\ref{smoothK}, let $u\in \cC^\infty_K(X,\Delta_1,L)$. Assume that there exists a current $T\in \mathcal{D}_K(X,L)$ such that
				\begin{equation}\label{reg1}
					u=\mathcal{D}_K T + T_{\Delta_1},
				\end{equation}
				where $T_{\Delta_1}\in \mathcal{D}_K(X,L)$ is a order-zero current supported on $\Delta_1$.
				Then there exists $v\in \cC^\infty_K(X,\Delta_1,L)$ such that
				\begin{equation}\label{reg2}
					u = D_K v \qquad \text{on } X\setminus \Delta_1.
				\end{equation}
			\end{lemme}
			
			\bigskip
			
			\subsection{A log canonical vanishing theorem}
			
			Finally, as an application of the theory established above, we prove the following vanishing theorem, which will be very useful.
			
			\begin{thm}\label{ddbar01}\cite[Thm 1.1]{JCMP}
				Let $X$ be an $n$-dimensional compact K\"ahler manifold and let $D=\Delta_1+\Delta_2$ be an snc divisor on $X$.
				Let $s_{\Delta_1}$ be the canonical section of $\Delta_1$.
				Let $(L,h_L)$ be a holomorphic line bundle on $X$ such that
				\[
				i\Theta_{h_L}(L)=\sum_{D_i\subset \Delta_2} a_i [D_i]+\theta,
				\]
				where $a_i\in ]0,1[\cap \mathbb Q$ and $\theta\ge 0$ is a smooth semipositive form.
				
				Let $\lambda$ be a $\dbar$-closed smooth $(n,q)$-form with values in $L+\Delta_1$. If there exist two $L$-valued forms $\beta_1$ and $\beta_2$ with logarithmic poles along $D$ such that
				\begin{equation}\label{relation}
					\frac{\lambda}{s_{\Delta_1}}
					= D'_{h_L}\beta_1 + \theta \wedge \beta_2
					\qquad \text{on } X\setminus D,
				\end{equation}
				then $\lambda$ is $\dbar$-exact, i.e.\
				\[
				[\lambda]=0 \in H^q\bigl(X,K_X+L+\Delta_1\bigr).
				\]
			\end{thm}
			
			We present here a proof which is slightly different from the original one. In the original proof, the key lemma \cite[Lemma 4.3]{JCMP} is proved by induction together with properties of the Green operator and Hodge decomposition. Here we give a different proof of \cite[Lemma 4.3]{JCMP} by using the $\ddbar$-lemma on every stratum of $\Delta_1$. This is the idea used in \cite[Section 10.2]{CDHP}.
			
			\begin{proof}
				Let $\varphi \in C^\infty_{(0,n-q)}\bigl(X,-L-\Delta_1\bigr)$ be a test form with compact support in $X\setminus Y$.
				Let $\omega_{\cC}$ be a conic metric associated with $(L,h_L)$, and equip $\Delta_1$ with an arbitrary smooth metric $h_1$.
				By Theorem~\ref{CoHo1}, we have a conic Hodge decomposition for $(L+\Delta_1, h_L\cdot h_1,\omega_{\cC})$:
				\begin{equation}\label{van3}
					\varphi=\varphi_1+\varphi_2,
					\qquad \varphi_1\in \ker \dbar,\quad \varphi_2\in (\ker \dbar)^\perp.
				\end{equation}
				
				\medskip
				
				We claim that
				\begin{equation}\label{check1}
					\int_X \lambda\wedge \varphi_1 = 0
				\end{equation}
				and that there exists a constant $C>0$ such that
				\begin{equation}\label{check2}
					\int_X |\varphi_2|^2 \,\omega_{\cC}^n
					\le C \int_X |\dbar \varphi_2|^2\,\omega_{\cC}^n
					= C \int_X |\dbar \varphi|^2\,\omega_{\cC}^n.
				\end{equation}
				Once \eqref{check1} and \eqref{check2} are established, standard functional analytic arguments imply that $[\lambda]=0 \in H^q(X,K_X+L+\Delta_1)$.
				
				\medskip
				
				The estimate \eqref{check2} follows from the closedness of $\im \Delta''$ in the conic Hodge decomposition theorem \ref{CoHo1}. 
				We now explain the proof of \eqref{check1}.
				
				Suppose $\Delta_1=\sum_{i=1}^m E_i$. Set
				\[
				E_{i_1,\dots,i_k}:=E_{i_1}\cap\cdots\cap E_{i_k},
				\]
				and let $\Res_{i_1,\dots,i_k}$ denote the residue on the stratum $E_{i_1,\dots,i_k}$.
				Taking residues in \eqref{relation} yields
				\begin{equation}\label{resiequl}
				\Res_{i_1,\dots,i_k}\!\left(\frac{\lambda}{s_{\Delta_1}}\right)
				= D'_{h_L}\Res_{i_1,\dots,i_k}(\beta_1)
				+ \theta \wedge \Res_{i_1,\dots,i_k}(\beta_2)
				\qquad \text{on } E_{i_1,\dots,i_k},
				\end{equation}
				which is an $L$-valued conic current on $E_{i_1,\dots,i_k}$.
				
				We claim that for every multi-index $(i_1,\dots,i_k)$,
				\begin{equation}\label{clima}
					\Res_{i_1,\dots,i_k}\!\left(\frac{\lambda}{s_{\Delta_1}}\right)
					= \dbar R_{i_1,\dots,i_k} + \sum_j [R_{j,i_1,\dots,i_k}],
				\end{equation}
				where $R_{i_1,\dots,i_k}$ are order-zero currents on $E_{i_1,\dots,i_k}$, orthogonal to the harmonic forms on $E_{i_1,\dots,i_k}$, and satisfying the skew-symmetry relations.  Here, we say that a $L$-valued coninc current $T$ on $E_{i_1,\dots,i_k}$ is orthogonal to the harmonic part, if 
				$$\int_{E_{i_1,\dots,i_k}} T \wedge \xi =0$$
				for every $L^\star$-valued harmonic form $\xi$ on $E_{i_1,\dots,i_k}$.
				
				\medskip
				
				We prove this claim by induction, starting from the deepest stratum. First,
				\[
				\Res_{1,\dots,m}\!\left(\frac{\lambda}{s_{\Delta_1}}\right)
				= D'_{h_L}\Res_{1,\dots,m}\beta_1
				+ \theta \wedge \Res_{1,\dots,m}\beta_2.
				\]
				The left-hand side is $\dbar$-closed. By the Bochner identity, the right-hand side is orthogonal to the harmonic part. Hence, by Theorem~\ref{conicddabr},
				\[
				\Res_{1,\dots,m}\!\left(\frac{\lambda}{s_{\Delta_1}}\right)=\dbar R_{1,\dots,m},
				\]
				for some $R_{1,\dots,m}$ on $E_{1,\dots,m}$ orthogonal to the harmonic part.
				
				Assume by induction that \eqref{clima} holds at level $k$. For level $k-1$,
				\begin{align*}
					\dbar \Res_{i_1,\dots,i_{k-1}}\!\left(\frac{\lambda}{s_{\Delta_1}}\right)
					&= \sum_j \Res_{j,i_1,\dots,i_{k-1}}\!\left(\frac{\lambda}{s_{\Delta_1}}\right)
					= \sum_j [\dbar R_{j,i_1,\dots,i_{k-1}}],
				\end{align*}
				where the last equality uses the skew-symmetry relations. Therefore the conic current
				\begin{equation}\label{orthtwo}
				\Res_{i_1,\dots,i_{k-1}}\!\left(\frac{\lambda}{s_{\Delta_1}}\right)
				- \sum_j [R_{j,i_1,\dots,i_{k-1}}]
				\end{equation}
				is $\dbar$-closed. Moreover, we can show that both two terms in \eqref{orthtwo} are orthogonal to the harmonic forms on $E_{i_1,\dots,i_{k-1}}$ as follows:
				
				 By using \eqref{resiequl}, the first term of \eqref{orthtwo} is orthogonal to the harmonic forms. To treat the second term of \eqref{orthtwo}, we use the following fact. As $\theta \geq 0$, by using Bochner equality, we can show that if $\xi$ is a $L^\star$-valued $\Delta''$-harmonic $(0,q)$ on $E_{i_1,\dots,i_{k-1}}$, the restriction $\xi |_{E_{j,i_1,\dots,i_{k-1}}}$ is still $\Delta''$-harmonic. 
				 
				 Consequently, $\Res_{i_1,\dots,i_{k-1}}\!\left(\frac{\lambda}{s_{\Delta_1}}\right)
				 - \sum_j [R_{j,i_1,\dots,i_{k-1}}]$ is $\dbar$-exact and we can write
				\[
				\Res_{i_1,\dots,i_{k-1}}\!\left(\frac{\lambda}{s_{\Delta_1}}\right)
				= \dbar R_{i_1,\dots,i_{k-1}} + \sum_j [R_{j,i_1,\dots,i_{k-1}}],
				\]
				for some $R_{i_1,\dots,i_{k-1}}$ orthogonal to the harmonic forms; moreover, the skew-symmetry relations can be arranged. This completes the induction.
				
				\medskip
				
				Applying \eqref{clima} with $k=0$, we conclude that $\frac{\lambda}{s_{\Delta_1}}$ is $\dbar$-exact in the conic current sense modulo a order-zero current supported on $\Delta_1$. Therefore,
				\[
				\int_X \lambda\wedge \varphi_1
				= \int_X \frac{\lambda}{s_{\Delta_1}} \wedge \bigl(s_{\Delta_1}\cdot \varphi_1\bigr)
				=0,
				\]
				which proves \eqref{check1}.
			\end{proof}

\section{Applications, I: extension of the pluri-canonical forms}\label{extenpluri}

In the rest of this article, we discuss several applications of the twisted Hodge theory established above. In this section, we work in the following setting.

Let $p:\cX\to \DD$ be a holomorphic family of smooth compact manifolds over the unit disc $\DD$. We assume that the central fiber, denoted by $X$, is K\"ahler. We denote by $K_\cX$ the canonical bundle of $\cX$, and let $\Lie \to \cX$ be an arbitrary holomorphic line bundle.

We are interested in extension properties of global sections of the adjoint bundle $K_\cX+\Lie$ defined over the $k^{\rm th}$ infinitesimal neighborhood of the central fiber of $p$, where $k\ge 1$ is a positive integer. In other words, we consider the sheaves
\begin{equation}\label{intr1}
	\cF_k := (K_\cX+\Lie)\otimes \O_\cX / t^{k+1}\O_\cX,
\end{equation}
and the corresponding spaces of sections $H^0(\cX,\cF_k)$. Here $t$ is the coordinate on $\DD$. The main question is as follows.

\begin{quest}
	Given $s\in H^0(\cX,\cF_k)$, can we extend $s$ to the $(k+1)^{\rm th}$ infinitesimal neighborhood of the central fiber $X$? In other words, does there exist $\wh s\in H^0(\cX,\cF_{k+1})$ such that $\pi_k(\wh s)=s$, where $\pi_k$ is the natural projection
	\begin{equation}\label{intr2}
		\pi_k:\cF_{k+1}\to \cF_k?
	\end{equation}
\end{quest}

\medskip

We first consider the case
\[
\Lie = (m-1)K_\cX,
\]
where $m\ge 1$ is a positive integer. Set $L:=(m-1)K_X$. Then the restriction of $s$ to the central fiber $X$ induces a metric on $L$, denoted by $h_L$. Let $\varphi_L$ be the weight of $h_L$. Then
\[
\varphi_L \sim \frac{m-1}{m}\log |s|_X|^2.
\]

\medskip

\noindent
The following statement is the main result of \cite{CP-inv}.

\begin{thm}\label{MT}\cite[Thm 1.1]{CP-inv}
	Let $s\in H^0(\cX,\cF_k)$ with $\Lie=(m-1)K_\cX$. Assume that $s$ admits a $C^\infty$ extension $s_k$ such that, writing
	\[
	\dbar s_k = t^{k+1}\Lambda_k,
	\]
	the integral
	\begin{equation}\label{intr7}
		\int_X \left|\frac{\Lambda_k}{dt}\right|^2 e^{-(1-\ep)\varphi_L}\, dV <  + \infty
	\end{equation}
	converges for every $\ep>0$.
	Then there exists a section $\wh s\in H^0(\cX,\cF_{k+1})$ such that $\pi_k(\wh s)=s$.
\end{thm}

One of the key ingredients in the proof of Theorem~\ref{MT} is the log vanishing Theorem~\ref{ddbar01}. We now sketch the proof of Theorem~\ref{MT}.

\begin{proof}[Sktech of the proof]
	Using a partition of unity, we can always construct a $C^\infty$ extension $s_k$ such that $\dbar s_k = t^{k+1}\Lambda_k$ for some smooth form $\Lambda_k$.
	Set $\lambda_k := (\Lambda_k/dt)|_X$. Since $\Lambda_k$ is $\dbar$-closed, $\lambda_k$ is also $\dbar$-closed. To prove the theorem, it suffices to show that $\lambda_k$ is $\dbar$-exact.
	
	Note that $s_k$ induces a Lie derivative as follows. The section $s_k$ defines an operator
	\[
	D' : C^\infty_{p,q}(\cX \setminus \Div(s_k),\Lie)\to C^\infty_{p+1,q}(\cX \setminus \Div(s_k),\Lie),
	\]
	locally given on a chart $\Omega$ by
	\[
	D'|_{\Omega} = \partial - \frac{m-1}{m}\frac{\partial s_k}{s_k}\wedge \cdot .
	\]
	We fix a horizontal vector field $\Xi \in C^\infty(\cX,T_\cX)$ such that $p_\ast(\Xi)=\partial/\partial t$. We then define the Lie derivative by
	\[
	\Lie_\Xi := [D',\Xi].
	\]
	Thus,
	\[
	\Lie_\Xi : C^\infty_{n+1,q}(\cX \setminus \Div(s_k),\Lie)\to C^\infty_{n+1,q}(\cX \setminus \Div(s_k),\Lie).
	\]
	
	\medskip
	
	Applying $\Lie_\Xi$ $(k+1)$ times to the equation $\dbar s_k = t^{k+1}\Lambda_k$, one can show that
	\[
	\lambda_k = \dbar \alpha + D'_{h_L}\beta
	\qquad \text{on } X\setminus (s=0),
	\]
	where $\alpha$ and $\beta$ are forms that are meromorphic along $s=0$ (possibly with poles of higher order).
	
	\medskip
	
	After desingularization and using \eqref{intr7}, we may assume that:
	\begin{itemize}
		\item The support of $\Div(s)$ can be written as $D+E$, where $D+E$ is snc, with $D=\sum D_i$ and $E=\sum E_i$.
		\item $\lambda_k$ is an $(F+E)$-valued $(n,1)$-form, where $F\equiv \sum \delta_i D_i$ for $\delta_i \in ]0,1[\cap \mathbb Q$.
	\end{itemize}
	
	\medskip
	
	Let $h_F$ be a metric on $F$ such that
	\[
	\frac{i}{2\pi}\Theta_{h_F}(F)=\sum \delta_i [D_i].
	\]
	Using the singularities of $h_F$, we can eliminate the higher-order poles of $\alpha$ and $\beta$. We obtain the following statement.
	
	There exist forms $\alpha_1$ and $\beta_1$ of types $(n,0)$ and $(n-1,1)$, respectively, with values in $F$ and with logarithmic poles along $E$, such that
	\[
	\dbar \alpha_1 + D'_{h_F}\beta_1 = \frac{\lambda_k}{s_E}
	\qquad \text{on } X\setminus E.
	\]
	By Theorem~\ref{ddbar01}, it follows that $\lambda_k$ is $\dbar$-exact.
\end{proof}

\begin{remark}\label{1orderex}
	For $k=1$, one can always find a smooth extension $s_1$ such that the integrability condition \eqref{intr7} is satisfied; see \cite[Thm 4.1]{CP-inv}.
\end{remark}

\noindent
Theorem~\ref{MT} is also motivated by an important conjecture in K\"ahler geometry, which we now recall. After Y.-T.~Siu's ``invariance of plurigenera'' papers \cite{Siu98,Siu00} (see also \cite{Pau07}), which concern the extension of $s$ in the case where $\cX$ is a \emph{projective} family, the following fundamental problem remains open.

\begin{conjecture}\label{Siuconj}\cite{Siu00}
	Let $p:\cX\to \DD$ be a family of smooth, $n$-dimensional compact manifolds, and let $X$ be the central fiber. Assume that the total space $\cX$ admits a K\"ahler metric. Then any holomorphic pluricanonical section defined on $X$ extends holomorphically to $\cX$.
\end{conjecture}

\smallskip

\noindent
Among the works devoted to Conjecture~\ref{Siuconj}, we first mention \cite{Lev83,Lev85} by M.~Levine. A direct consequence of Theorem~\ref{MT} is the following statement, which can be viewed as a more general version of the results in \cite{Lev83,Lev85}.

\begin{cor}\label{corlev}\cite[Cor 1.5]{CP-inv}
	Let $p:\cX\to \DD$ be a smooth holomorphic family of compact manifolds whose central fiber $X$ is K\"ahler, and let $s\in H^0(X,mK_X)$ be a pluricanonical section on $X$.
	Assume that the zero set of the ideal
	\[
	\mathfrak I := \lim_{\ep\to 0}\mathcal I\!\left((1-\ep)\frac{m-1}{m}\Sigma\right)
	\]
	is discrete, where $\Sigma$ is the divisor corresponding to $s$.
	Then any section $\tau \in H^0(X,mK_X)$ admits a holomorphic extension to a neighborhood of the central fiber $X$.
	If moreover the total space $\cX$ is K\"ahler, then any section $\tau \in H^0(X,mK_X)$ admits a holomorphic extension to $\cX$.
\end{cor}

\begin{remark}\cite{CP-inv}
	Thanks to Remark~\ref{1orderex} and Theorem~\ref{MT}, we can always extend $s$ to the first-order infinitesimal neighborhood of $X$.
\end{remark}

\section{Applications, II: cohomology jumping loci property. }\label{jumping}

In this section, we discuss another type of extension problem: the base manifold is fixed while the line bundle varies. In this setting, the slogan is that extension to first order implies extension to a genuine neighborhood. To begin, we recall a classical result.

Let $X$ be a compact K\"ahler manifold with a holomorphic line bundle $(L,\dbar_L)$, and let $\alpha$ be a smooth $\dbar$-closed $(0,1)$-form on $X$. For each $t \in (\CC,0)$, consider the twisted $\dbar$-operator
\begin{equation}\label{intr1}
	\dbar_t := \dbar_L + t\alpha
\end{equation}
acting on the space of smooth $L$-valued $(p,q)$-forms on $X$.
A natural problem is to study how the finite-dimensional spaces
\begin{equation}\label{intr22}
	\ker(\dbar_t)\big/\im(\dbar_t)
\end{equation}
vary with $t$. General results (see \cite{Kod86}) ensure that the dimension of the spaces in \eqref{intr22} defines an upper semicontinuous function of $t$.

A starting point for results of this type is due to Green--Lazarsfeld \cite{GL87,GL91}.
Let $u \in \ker \dbar_L$. We say that $u$ \emph{extends to first order in the direction $\alpha$} if the wedge product $u\wedge \alpha$ is $\dbar$-exact, i.e.\ if we can solve
\begin{equation}\label{intr3}
	\dbar v = u\wedge \alpha.
\end{equation}
Following \cite[\S 1]{GL87}, we have the following extension property.

\begin{thm}\label{extencla}
	Let $X$ be a compact K\"ahler manifold and let $(L,\dbar_L)$ be the trivial line bundle.
	Let $\alpha$ be a smooth $\dbar$-closed $(0,1)$-form on $X$.
	Let $u$ be a smooth $\dbar$-closed $(p,q)$-form on $X$ which extends to first order in the direction $\alpha$.
	Then there exists a smooth family $(u_t)_{t\in (\CC,0)}$ of $(p,q)$-forms such that
	\begin{equation}\label{intr4}
		u_0 = u, \qquad u_t \in \ker \dbar_t.
	\end{equation}
	In other words, extension to first order implies extension to a genuine neighborhood.
\end{thm}

\begin{proof}
	We aim to solve the following system:
	\begin{equation}\label{simp46}
		\begin{cases}
			\dbar(u_1) + \alpha \wedge u = 0,\\
			\dbar(u_2) + \alpha \wedge u_1 = 0,\\
			\ldots\\
			\dbar(u_{i+1}) + \alpha \wedge u_i = 0,\\
			\ldots
		\end{cases}
	\end{equation}
	Note that $u$ extends to first order along $\alpha$ if and only if the first equation in \eqref{simp46} can be solved.
	Similarly, $u$ extends to arbitrary order if and only if all equations in \eqref{simp46} can be solved.
	
	Let $u_0$ be a harmonic representative of $u$. Since the first equation in \eqref{simp46} is solvable by assumption, we have
	\[
	\alpha \wedge u_0 \in \im \dbar \cap \ker \partial.
	\]
	By the $\partial\dbar$-lemma, there exists $u_1 \in \im \partial$ such that
	\[
	\dbar(u_1) + \alpha \wedge u_0 = 0.
	\]
	
	To solve the second equation, note that since $u_1 \in \im \partial$, we have $\alpha \wedge u_1 \in \im \partial$. Moreover,
	\[
	\dbar(\alpha \wedge u_1) = - \alpha \wedge \dbar u_1
	= \alpha \wedge \alpha \wedge u_0 = 0.
	\]
	Hence $\alpha \wedge u_1 \in \im \partial \cap \ker \dbar$.
	By the $\partial\dbar$-lemma again, there exists $u_2 \in \im \partial$ such that
	\[
	\dbar(u_2) + \alpha \wedge u_1 = 0.
	\]
	The same argument applies inductively to the remaining equations.
	
	Finally, if each $u_i$ is chosen as the minimal solution of \eqref{simp46}\footnote{It is then automatically $\partial$-exact by the Hodge decomposition theorem.},
	one can prove convergence of the series $\sum_{i=0}^\infty z^i u_i$ for $|z|\ll 1$.
	In particular, $u$ extends to a genuine neighborhood.
\end{proof}

\noindent
As an application of Theorem~\ref{extencla}, we have the following jumping locus property.

\begin{thm}\cite{GL87, GL91}\label{gl}
	Let $X$ be a compact K\"ahler manifold. Define the jumping locus
	\[
	\Sigma^{p,q}_k(X)
	:= \{ L \in \Pic^0(X) \mid \dim H^{q}(X,\Omega_X^{p}\otimes L) \ge k \}.
	\]
	Then $\Sigma^{p,q}_k(X)$ is a finite union of translates of subtori in $\Pic^0(X)$.
\end{thm}

Theorem~\ref{gl} has been generalized in several directions. In the compact K\"ahler setting, \cite{Sim93} shows moreover that $\Sigma^{p,q}_k(X)$ is in fact a finite union of torsion translates of subtori in $\Pic^0(X)$. More recently, \cite{BW15,BW20} generalize Theorem~\ref{gl} to the quasi-compact K\"ahler case. Their proof relies on global considerations involving moduli spaces and the Riemann--Hilbert correspondence (with some ideas tracing back to \cite{Sim93}).

In \cite{CDHP}, using Hodge theory for rank-one local systems on quasi-compact K\"ahler manifolds, we give a new proof of the jumping locus property in the quasi-compact case. The idea is to use the $\ddbar$-lemma as in Theorem~\ref{extencla}. Thus our proof is self-contained and differs substantially from the original arguments in \cite{BW15,BW20}. More precisely, we have the following result.

\begin{thm}\cite{BW20,CDHP}\label{jmp}
	Let $X$ be a compact K\"ahler manifold and let $D=\sum_{i=1}^{k} D_i$ be a simple normal crossing divisor. Then the jumping locus $\Sigma_k^i(X\setminus D)_{\rm B}$ is a finite union of translates of subtori.
\end{thm}

\begin{remark}
	Conjecturally, $\Sigma_k^i(X\setminus D)_{\rm B}$ is a finite union of torsion translates of subtori. This is proved in the quasi-projective case in \cite{BW15}.
\end{remark}

We now briefly outline the main idea of the proof of Theorem~\ref{jmp}. First, by analogy with the extension theorem (Theorem~\ref{extencla}), we prove an extension theorem for logarithmic holomorphic flat line bundles.

\begin{thm}\label{solveequs}
	In the setting of Section~\ref{hodgeloc}, set $\Delta_1 := \sum_{i=1}^{\ell} Y_i$.
	Let $(L,\nabla) \in M_{DR}(X/D)$ and let
	\[
	(\xi,\beta)\in \mathcal{H}^{0,1} \oplus H^0\bigl(X,\Omega_X^1(\log D)\bigr)
	= (T_{M_{DR}(X/D)})_{(L,\nabla)}.
	\]
	Let $u_0 \in \cC^\infty_{\bullet,K}(X,\Delta_1,L)$ such that
	\[
	D_K u_0 = 0 \qquad \text{on } X\setminus D.
	\]
	Suppose that the following hold.
	\begin{itemize}
		\item (1) The form $u_0$ extends to first order, i.e.\ there exists $v_1\in \cC^\infty_{\bullet,K}(X,\Delta_1,L)$ such that
		\[
		D_K v_1 = -(\xi+\beta)\wedge u_0 \qquad \text{on } X\setminus D.
		\]
		\item (2) For every non-empty $I\subset \{1,\ldots,\ell\}$, any $D_K$-closed form in $\cC^\infty_{\bullet-|I|,K}(Y_I,L)$ extends to first order; that is, for any $s\in \cC^\infty_{\bullet-|I|,K}(Y_I,L)$, if $D_K s = 0$, then $-\iota_I^*(\xi+\beta)\wedge s$ is $D_K$-exact on $Y_I$.
		Here $Y_I := \cap_{i\in I} Y_i$ and $\iota_I:Y_I\hookrightarrow X$ is the inclusion map.
	\end{itemize}
	Then for any $k\ge 0$ we can solve
	\begin{equation}\label{simp4611}
		\begin{cases}
			D_K(u_1)= -(\xi+\beta)\wedge u_0,\\
			D_K(u_2)= -(\xi+\beta)\wedge u_1,\\
			\ldots\\
			D_K(u_{k+1})= -(\xi+\beta)\wedge u_k,
		\end{cases}
	\end{equation}
	on $X\setminus D$, with $u_i \in \cC^\infty_{\bullet,K}(X,\Delta_1,L)\cap \im D_K$, and such that the series $\sum_{i\ge 0} t^i u_i$ converges for $|t|\ll 1$.
\end{thm}

\begin{proof}[Sketch of proof]
	First assume that $\Delta_1=0$.
	
	Let $u_0$ be a $D_K$-closed form with values in $L$. After changing the representative, we may assume that $u_0$ is harmonic. Then $D_K^c u_0=0$.
	By the hypothesis $(1)$, we have
	\[
	(\xi+\beta)\wedge u_0 \in \im D_K \cap \ker D_K^c.
	\]
	By the $\ddbar$-lemma \ref{ddbarK}, there exists $u_1 \in \cC^\infty_{\bullet,K}(X,L)\cap \im D_K^c$ such that
	\begin{equation}\label{par3}
		D_K u_1 = -(\xi+\beta)\wedge u_0.
	\end{equation}
	
	Now we consider the second equation in \eqref{simp4611}.
	By Remark~\ref{realsmooth}, $(\xi+\beta)\wedge u_1 \in \cC^\infty_{\bullet+1,K}(X,L)$.
	Since $u_1$ is $D_K^c$-exact, $(\xi+\beta)\wedge u_1$ is also $D_K^c$-exact. Moreover, $(\xi+\beta)\wedge u_1$ is $D_K$-closed, because $D_K u_1$ is a multiple of $(\xi+\beta)$.
	By the $\ddbar$-lemma \ref{ddbarK}, there exists $u_2 \in \cC^\infty_{\bullet,K}(X,L)\cap \im D_K^c$ such that
	\begin{equation}\label{par8}
		D_K u_2 = -(\xi+\beta)\wedge u_1.
	\end{equation}
	Iterating this procedure, we obtain successively
	\[
	D_K u_m = -(\xi+\beta)\wedge u_{m-1}, \qquad u_m \in \im D_K^c,
	\]
	for all $m\ge 1$, i.e.\ $u_0$ extends to arbitrary order.
	
	\medskip
	
	In the general case where $\Delta_1$ is not necessarily $0$, we first embed $\cC^\infty_{\bullet,K}(X,\Delta_1,L)$ into the current space $\mathcal{D}_K (X,L)$. In the sense of currents, using the current $\ddbar$-lemma \ref{ddbarcurrentconicK} and induction on every stratum of $\Delta_1$, we can solve \eqref{simp4611} modulo currents supported on $\Delta_1$. The assumpition $(2)$ is used in this place. 
	Finally, thanks to the regularity lemma \ref{regularprop}, we can solve \eqref{simp4611} in $\cC^\infty_{\bullet,K}(X,\Delta_1,L)$. We refer to \cite[Section 10]{CDHP} for details.
\end{proof}

As an application of Theorem~\ref{solveequs}, we study the jumping locus property on quasi-compact K\"ahler manifolds. Let $X$ be a compact K\"ahler manifold and let $D\subset X$ be an snc divisor. Set
\[
\Sigma_k^i(X/D)_{\rm DR}
:= \{ (L,\nabla)\in M_{DR}(X/D) \mid \dim \mathbb H^i(X,\Omega_X^\bullet(\log D)\otimes L)\ge k\}.
\]
By combining Theorem~\ref{solveequs} with Theorem~\ref{quasiiso}, we obtain the following.

\begin{cor}\label{nearbygeneric}
	Let $(L,\nabla)\in \Sigma_k^i(X/D)_{\rm DR}$ be a generic point, and suppose that
	\[
	\Rel\bigl(\Res_{D_j}(L,\nabla)\bigr) \in [0,1[
	\]
	for every irreducible component $D_j$ of $D$.
	Let $\mathbb D$ be the unit disc in $\CC$, and consider a holomorphic map
	\[
	\gamma:\mathbb D \to \Sigma_k^i(X/D)_{\rm DR}
	\]
	such that $\gamma(0)=(L,\nabla)$.
	Let $\gamma'(0)\in \mathcal H^{0,1}(X)\oplus H^0\bigl(X,\Omega_X(\log D)\bigr)$ denote the first-order deformation of $\gamma$.
	Then
	\begin{equation}\label{uppertorus}
		\dim \mathbb H^i\!\left(X,\Omega_X^\bullet(\log D)\otimes L_{\gamma(0)+z\gamma'(0)}\right)\ge k
	\end{equation}
	for every $|z|\ll 1$.
	Here $\Omega_X^\bullet(\log D)\otimes L_{\gamma(0)+z\gamma'(0)}$ is the logarithmic de Rham complex induced by $\gamma(0)+z\gamma'(0)\in M_{\rm DR}(X/D)$.
\end{cor}

As a consequence, we obtain the jumping locus property for rank-one local systems on $X\setminus D$.

\begin{thm}\label{complete2}
	Let $X$ be a compact K\"ahler manifold and let $D=\sum_{i=1}^{k} D_i$ be a simple normal crossing divisor.
	Let $M'_{\rm B}(X/D) \subset M_{\rm B}(X/D)$ be a connected component. Then the jumping locus
	\[
	\Sigma_k^i(X\setminus D)_{\rm B}
	:= \{ \tau \in M'_{\rm B}(X/D) \mid \dim H^i(X\setminus D,\tau) \ge k\}
	\]
	is a finite union of translates of subtori.
\end{thm}

\begin{proof}
	Pick a general element $\tau\in \Sigma_k^i(X\setminus D)_{\rm B}$.
	Let $(L,\nabla)\in M_{\rm DR}(X/D)$ be such that $\tau = \rho(L,\nabla)$ and $\Rel(\Res(L,\nabla))\in [0,1[$ for every component of $D$.
	
	Then for any $(L',\nabla')$ in a sufficiently small neighborhood of $(L,\nabla)$, the residues are not strictly positive integers along any component of $D$.
	Therefore, by Theorem~\ref{Deli}, we have
	\[
	\dim \mathbb H^{i}\bigl(X,\Omega_X^\bullet(\log D)\otimes L'\bigr)
	= \dim H^{i}\bigl(X\setminus D,\rho(L',\nabla')\bigr).
	\]
	Hence there exist neighborhoods $U$ of $(L,\nabla)$ and $V$ of $\tau$ such that
	\[
	\rho:\ \Sigma_k^i(X/D)_{\rm DR}\cap U \longrightarrow \Sigma_k^i(X\setminus D)_{\rm B}\cap V
	\]
	is a biholomorphism.
	We then conclude from Corollary~\ref{nearbygeneric} that $\Sigma_k^i(X\setminus D)_{\rm B}$ is a finite union of translates of subtori in $M'_{\rm B}(X\setminus D)$.
\end{proof}

	\section{Applications, III: deformation of log Calabi--Yau manifolds}\label{logCY}
	
	In this last section, we would like to present a nice application of conic Hodge theory to the deformation theory of log Calabi--Yau Kähler manifolds, proved recently by R.~Zhang \cite{Zhang}.
	
	In \cite{KKP}, L.~Katzarkov, M.~Kontsevich, and T.~Pantev proved that the locally trivial deformations of a log Calabi--Yau projective pair $(X,D)$ are unobstructed. In his recent paper \cite{Zhang}, using conic Hodge theory, R.~Zhang generalized the theorem of Katzarkov--Kontsevich--Pantev to the Kähler setting. More precisely, R.~Zhang proved the following result.
	
	\begin{thm}\cite[Thm A]{Zhang}\label{unobstr}
		Let $X$ be a compact Kähler manifold, and let
		\[
		D = \sum_{i=1}^s D_i
		\]
		be a simple normal crossing divisor on $X$. Assume that there exists a collection of
		weights $\{a_i\}_{1 \le i \le s} \subset [0,1] \cap \mathbb{Q}$ such that
		\begin{equation}\label{eq:logCY}
			c_1(K_X) + \sum_{i=1}^s a_i c_1(D_i) = 0 
			\quad \text{in } H^2(X,\mathbb{Q}).
		\end{equation}
		Then the locally trivial (infinitesimal) deformations of the pair $(X,D)$ are unobstructed. In particular, the pair $(X,D)$ admits a semi-universal (Kuranishi) family over a smooth base.
	\end{thm}
	
	Note that some special cases of Theorem \ref{unobstr} are proved in \cite{LRW, Wan18}. 
	As explained in \cite{Zhang} and \cite{Wan18}, Theorem~\ref{unobstr} can be derived from the following vanishing theorem.
	
	\begin{thm}\label{logddbar}\cite[Thm B]{Zhang}
		Let $X$ be a compact Kähler manifold, and let $D = \sum D_i$ be a simple normal crossing divisor on $X$. 
		Let $L$ be a holomorphic line bundle on $X$ such that
		\[
		c_1(L) = \sum_i a_i D_i \in H^2(X,\mathbb{Q})
		\]
		for some $a_i \in [0,1] \cap \mathbb{Q}$.
		
		Let $h_L$ be a metric on $L$ such that
		\[
		i\Theta_{h_L} = \sum a_i [D_i].
		\]
		Let $h_{L^\star}$ be the dual metric on $L^\star$.
		Let $\alpha \in A^{p,q}\bigl(X,D,L^\star\bigr)$ satisfy
		\[
		\dbar D'_{h_{L^\star}} \alpha = 0
		\quad \text{on } X \setminus D.
		\]
		Then there exists
		\[
		\chi \in A^{p,q-1}\bigl(X,D,L^\star\bigr)
		\]
		such that
		\begin{equation}\label{eq:logpartial}
			\dbar \chi = D'_{h_{L^\star}} \alpha
			\quad \text{on } X \setminus D.
		\end{equation}
	\end{thm}
	
	The original proof of Theorem~\ref{logddbar} uses conic Hodge decomposition theory together with J.~King's theory \cite{Kin71,Kin83} on the sheaf of currents annihilating null forms on $D$. Here we present a slightly different approach, replacing King’s theory with the regularity lemma Corollary~\ref{regcor}. 
	
	\begin{proof}[Proof of Theorem~\ref{logddbar}]
		We decompose
		\[
		D = \Delta_1 + \Delta_2,
		\]
		where $\Delta_1 := \sum_{i=1}^m D_i$ consists of the components such that $a_i = 0$, and $\Delta_2$ consists of the components such that $a_i > 0$.
		Then $\alpha \in \mathcal{C}^\infty(X,\Delta_1,L^\star)$. By Example~\ref{imporantex}, $\alpha \in \mathcal{D}(X,L^\star)$ can be viewed as an $L^\star$-valued current in the conic sense.
		
		We consider the conic current
		\[
		T := D'_{h_{L^\star}} \alpha \in \mathcal{D}(X,L^\star).
		\]
		Note that for every subset $\{i_1,\dots,i_k\} \subset \{1,\dots,m\}$, we have
		\begin{equation}\label{residuestra}
			\Res_{i_1,\dots,i_k} T
			= D'_{h_{L^\star}} \Res_{i_1,\dots,i_k} \alpha
			\quad \text{on } D_{i_1} \cap \cdots \cap D_{i_k}.
		\end{equation}
		
		We claim that for every subset $\{i_1,\dots,i_k\} \subset \{1,\dots,m\}$, there exist conic currents $R_{i_1,\dots,i_k}$ (resp.~$R_{j,i_1,\dots,i_k}$) supported on $D_{i_1}\cap\cdots\cap D_{i_k}$ of order $0$ (resp.~$D_j\cap D_{i_1}\cap\cdots\cap D_{i_k}$), satisfying the skew-symmetry relations, such that
		\begin{equation}\label{claimfin}
			\Res_{i_1,\dots,i_k} T
			= \dbar D'_{h_{L^\star}} R_{i_1,\dots,i_k}
			+ \sum_{j=1}^m [D'_{h_{L^\star}} R_{j,i_1,\dots,i_k}]
			\in \mathcal{D}(D_{i_1}\cap\cdots\cap D_{i_k},L^\star).
		\end{equation}
		
		We prove the claim by induction on $k$, starting from $k=m$.
		
		For $k=m$, by \eqref{residuestra}, we have
		\[
		\Res_{1,\dots,m} T
		= D'_{h_{L^\star}} \Res_{1,\dots,m} \alpha
		\]
		in the sense of conic currents on $D_1 \cap \cdots \cap D_m$.
		The left-hand side is $\dbar$-exact, and the right-hand side is $D'_{h_{L^\star}}$-exact. By the $\ddbar$-lemma for conic currents (Lemma~\ref{ddbarcurrentconic}), we obtain
		\[
		\Res_{1,\dots,m} T
		= \dbar D'_{h_{L^\star}} R_{1,\dots,m}
		\]
		for some current $R_{1,\dots,m}$ on $D_1 \cap \cdots \cap D_m$ of order $0$, which can be chosen to satisfy the skew-symmetry relations. Thus, the claim holds for $k=m$.
		
		\medskip
		
		Assume by induction that \eqref{claimfin} holds for $k$. For the $(k-1)$-index case, we have
		\[
		\dbar \Res_{i_1,\dots,i_{k-1}} T
		= \sum_{j=1}^m \Res_{j,i_1,\dots,i_{k-1}} T
		\]
		\[
		= \sum_{j=1}^m [\dbar D'_{h_{L^\star}} R_{j,i_1,\dots,i_{k-1}}]
		+ \sum_{s,j=1}^m [D'_{h_{L^\star}} R_{s,j,i_1,\dots,i_{k-1}}]
		= \sum_{j=1}^m [\dbar D'_{h_{L^\star}} R_{j,i_1,\dots,i_{k-1}}],
		\]
		where the last equality follows from the skew-symmetry of $R_{s,j,i_1,\dots,i_{k-1}}$.
		
		Thus, the conic current
		\[
		\Res_{i_1,\dots,i_{k-1}} T
		- \sum_{j=1}^m [D'_{h_{L^\star}} R_{j,i_1,\dots,i_{k-1}}]
		\]
		is $\dbar$-closed. Moreover, by \eqref{residuestra}, it is $D'_{h_{L^\star}}$-exact.
		The $\ddbar$-lemma for conic currents (Lemma~\ref{ddbarcurrentconic}) then implies that it is $\ddbar$-exact.
		Hence, there exist order zero currents $R_{i_1,\dots,i_{k-1}}$ on $D_{i_1}\cap\cdots\cap D_{i_{k-1}}$ such that
		\[
		\Res_{i_1,\dots,i_{k-1}} T
		= \dbar D'_{h_{L^\star}} R_{i_1,\dots,i_{k-1}}
		+ \sum_{j=1}^m [D'_{h_{L^\star}} R_{j,i_1,\dots,i_{k-1}}].
		\]
		We may further require that $R_{i_1,\dots,i_{k-1}}$ satisfy the skew-symmetry relations.
		This completes the induction and proves the claim.
		
		\medskip
		
		Applying the claim \eqref{claimfin} for $k=0$, we find conic currents $R_i$ supported on $D_i$ of order $0$ such that
		\[
		T - \sum_{i=1}^m [D'_{h_{L^\star}} R_i]
		\]
		is $\ddbar$-exact on $X$ in the conic current sense.
		By the regularity lemma (Corollary~\ref{regcor}), the theorem follows.
	\end{proof}

	\end{document}